\documentclass[a4paper,11pt,oneside  ]{article}	

\makeindex
\usepackage{hyperref}
\usepackage{graphicx}
\usepackage{enumerate}
\hypersetup{colorlinks, linkcolor=blue, urlcolor=red, filecolor=green, citecolor=blue}
\usepackage{amsmath}
\usepackage{amsthm}
\usepackage{amssymb}
\usepackage{authblk}
\usepackage{enumerate}
\usepackage{fullpage}
\usepackage{esint}
\usepackage{mathtools}

\theoremstyle{plain}
\newtheorem{theorem}{Theorem}[section]

\newtheorem{lemma}[theorem]{Lemma}
\newtheorem{proposition}[theorem]{Proposition}

\theoremstyle{definition}
\newtheorem{definition}[theorem]{Definition}

\theoremstyle{remark}
\newtheorem{remark}{Remark}

\newcommand\loc{{\operatorname{loc}}}

\title{Large Deviations Principle for the tagged empirical field of a general {interacting gas}}
\author{David Padilla-Garza }

\begin{document}

\maketitle

\begin{abstract}
    This paper deals with rare events in a general {interacting gas} at high temperature, by means of Large Deviations Principles. The main result is an LDP for the tagged empirical field, which features the competition of an energy term and an entropy term. The approach to proving this Large Deviations Principle is to first deduce one for the tagged empirical field of non-interacting particles at high temperature, and upgrade that result to interacting particle systems.   
\end{abstract}

\section{Introduction and motivation}

Consider $N$ particles that interact via a pair-wise interaction and are confined by an external potential. This is modeled by the Hamiltonian
\begin{equation}
\label{eq:hamilton}
    \mathcal{H}_{N}(X_{N}) = \sum_{i \neq j} g(x_{i}- x_{j}) + N \sum_{i=1}^{N} V(x_{i}),
\end{equation}
where $X_{N} \in \mathbb{R}^{d \times N},$ with $X_{N} = (x_{1},...x_{N})$, $V:  \mathbb{R}^{d } \to \mathbb{R}$ is the confining potential, and $g:  \mathbb{R}^{d}\to \mathbb{R}$ is the pair-wise interaction. In this paper, we will work mainly with particles in the hypercube $\mathbb{Q}^{d} = \left[ -\frac{T}{2}, \frac{T}{2} \right]^{d} $. Given that it is useful to work with point configurations in $\mathbb{R}^{d}$, the way to work in the hypercube will be to assume that the confining potential is infinite outside of  $\mathbb{Q}^{d}$. {The size of the cube $T$ will play no role except to affect some constants. However, we keep the parameter to emphasize that we may treat a cube of arbitrary size.}  

Consider the Gibbs measure $\mathbf{P}_{N,\beta}$ associated to $\mathcal{H}_{N}:$
\begin{equation}
\label{Gibbs}
    \mathrm d  \mathbf{P}_{N,\beta} (X_{N}) = \frac{1}{Z_{N, \beta}} \exp\left( - \beta    \mathcal{H}_{N}(X_{N}) \right) \, \mathrm d X_{N},
\end{equation}
where $\beta = \beta(N)$ is the inverse temperature, and
\begin{equation}
    Z_{N, \beta} = \int_{\mathbb{R}^{d \times N}} \exp\left( - \beta    \mathcal{H}_{N}(X_{N}) \right) \, \mathrm d X_{N}
\end{equation}
is the partition function. A system governed by \eqref{Gibbs} will be called a general interacting gas, or interacting particle system. 

A very frequent form of the pair-wise interaction $g$ is given by
\begin{equation}
\label{Coulombcase}
\begin{cases}
g(x)=\frac{1}{|x|^{d-2}} \text{ if } d \geq 3 \\
g(x)=-\log(|x|) \text{ if } d = 1,2
\end{cases}
\end{equation}
We will refer to this setting as the Coulomb case for $d \geq 2$, and log case for $d=1$. Another frequent form of $g$ is given by 
\begin{equation}\label{Rieszcase}
    g(x) = \frac{1}{|x|^{d-2s}},
\end{equation}
with $s \in \left(0,\min \left\{1, \frac{d}{2}\right\}\right)$. We will refer to this setting as the Riesz case. Coulomb and Riesz gases are a classical field with applications in Statistical Mechanics \cite{jancovici1993large, girvin2005introduction, stormer1999fractional}, Random Matrix Theory \cite{bourgade2014universality, bourgade2014local, hardy2021clt, lambert2019quantitative}, and Mathematical Physics \cite{berman2014determinantal, rougerie2015incompressibility}, among other fields. Although this paper does not deal with Coulomb or Riesz gasses, we mention them because many results for particle systems driven by the Gibbs measure (equation \eqref{Gibbs}) are for these kernels. Later on we will discuss some of them.

As mentioned earlier, this paper will not deal with a Coulomb or Riesz interaction. Instead, it will deal with a general interacting gas, a system of particles driven by the Gibbs measure (equation \eqref{Gibbs}) for general $g$. The study of a general {interacting gas} goes back decades \cite{georgii2011gibbs,ruelle1967variational,ruelle1968statistical}, and continues to attract attention \cite{garcia2019large, lambert2021poisson, chafai2014first,garcia2024generalized}. It is also linked to to AI and machine learning, more specifically to neural networks. Neural networks are common ways of accurately representing high-dimensional functions: given $f$ a high-dimensional function, it may be represented as
\begin{equation}
    f (x) := \lim_{n \to \infty} \frac{1}{n} \sum_{i=1}^{n} {\varphi}_{i} (x, \theta_{i}),
\end{equation}
for given basis functions ${\varphi}_{i} (x, \theta_{i})$, depending on parameters $\theta_{i}$. stochastic gradient descent (SGD) is one of the most popular algorithms to accurately tune the parameters $\theta_{i}$, but even despite its frequent use very few rigorous results for convergence existed for SGD until recently. This gap in the literature was addressed by \cite{rotskoff2018neural,rotskoff2018parameters,rotskoff2022trainability,wen2024coupling}, by modeling the evolution of parameters $\theta_{i}$ driven by the SGD algorithm as an interacting particle system, driven by a Stochastic PDE (SPDE). One advantage of this approach is that it allows the authors to prove a rate of convergence of order $O(n^{-1})$ to the mean-field limit. Additionally, the energy landscape for the empirical measure becomes convex, making it more amenable to the SGD algorithm. The precise form of the interaction kernel depends on the error between the measurements and the approximating function, and on the functions ${\varphi}_{i}$, but it is not given by a Coulomb or Riesz kernel except in very specific cases. Hence, as examples of general interactions, we may consider those that arise from modelling parameters as particles. This approach is linked to high dimensional particle systems as well (large $d$), since the function $f$ is high dimensional. 

Given that there are $N$ particles in space, and that typically they are confined to a compact set by the potential, the distance between the particles is of order $N^{-\frac{1}{d}}$. This means that the scale $N^{\frac{1}{d}}$ is special and will be called the \emph{microscopic scale}. We will call the original length scale \emph{macroscopic}. Any length scale which is between these two will be called \emph{mesoscopic}. 

At a macroscopic scale, the system is well-described by the empirical measure. Given $X_{N} = (x_{1}, ... x_{N}),$ the empirical measure is defined as 
\begin{equation}
\label{empm}
    {\rm emp}_{N} (X_{N}) := \frac{1}{N} \sum_{i=1}^{N} \delta_{x_{i}}.
\end{equation}
The behavior of systems governed by \eqref{Gibbs} can be significantly different depending on the order of magnitude of $\beta$. We will call the regime $\beta = \frac{\theta}{N}$ (with $\theta$ constant) the \emph{high-temperature} regime. In the regime $\frac{1}{N} \ll \beta$, the empirical measure converges (a.s. under the Gibbs measure) to a probability measure that typically has compact support, and is characterized by minimizing the mean-field limit of the Hamiltonian. This probability measure is called the equilibrium measure. In the high-temperature regime, however, this does not happen. Instead, the empirical measure converges (a.s. under the Gibbs measure) to a measure that is positive in the set in which $V$ is finite (in particular, positive a.e. if $V$ is finite a.e.). 

In the Coulomb or Riesz setting, we will call the regime $\beta = \theta N^{\frac{2s - d}{d}}$ the \emph{{low temperature}} regime. The reason for identifying this temperature scaling is that in the low-temperature regime, we observe structure at the microscopic level (in the sense that the tagged empirical field converges to a limit that is not Poisson, see \cite{leble2017large}), whereas at higher temperatures we don't (in the sense that the limit of the tagged empirical field is Poisson, see Proposition \ref{Prop:rieszmedtemp}). 

At a microscopic scale, the system is well-described by the local point process, which is defined (given a fixed $x \in \mathbb{R}^{d}$) as
\begin{equation}
    \sum_{i=1}^{N} \delta_{N^{\frac{1}{d}} (x_{i} - x)}.
\end{equation}
This object is a positive, discrete measure of mass $N$ on $\mathbb{R}^{d}$. Alternatively, the local point process may be thought of as an element of $\mathbb{R}^{d \times N}$, defined as $N^{\frac{1}{d}} \theta_{x} X_{N}$, where $\theta_{x} $ denotes the element-wise translation by the vector $x$. This object is often hard to study analytically (even in the non-interacting case), so we may define a less fine observable by averaging the local point process over a compact subset $\Omega \subset \mathbb{R}^{d}$:
\begin{equation}
\label{eq:nontagged}
   \frac{1}{|\Omega|} \int_{\Omega} \delta_{\left(N^{\frac{1}{d}} \theta_{x} X_{N}\right)}  \, \mathrm d x.
\end{equation}
This is called the empirical field, and it is a probability measure on the space of point configurations. A more precise observable, which is still not as precise as the local point process, is the tagged empirical field. It is given by averaging the local point process while keeping track of the blow-up point. The tagged empirical field is defined as 
\begin{equation}
    \frac{1}{|\Omega|} \int_{\Omega} \delta_{\left(x, N^{\frac{1}{d}} \theta_{x} X_{N}\right)}  \, \mathrm d x.
\end{equation}
The tagged empirical field yields a measure on the cross product of $\Omega$ and the space of point configurations. The first marginal of this measure is the normalized Lebesgue measure on $\Omega$, and the second marginal is the empirical field, given by equation \eqref{eq:nontagged}. 

The main goal of this paper is to derive a Large Deviations Principle (LDP) for the tagged empirical field at high temperature ($\beta = \frac{1}{N}$) in the setting of a general interaction. This problem is motivated from several angles:
\begin{itemize}
    \item[1.] It was proved in \cite{lambert2021poisson} that the local point process converges to a Poisson Point Process in the high-temperature regime for a very general class of interactions. A natural question is then to quantify this convergence by means of an LDP. {Ideally, we would like to understand rare events in the local point process}. Still, as mentioned before, this is a very complicated problem, so we can quantify this convergence with an LDP for the tagged empirical field. 
    \item[2.] In, \cite{leble2017large} the authors treat the Riesz and Coulomb cases at {low temperature} ($\beta = \theta N^{\frac{2s - d}{d}}$) \footnote{In \cite{leble2017large}, different units and also different notations are used. Hence the $\beta$ in \cite{leble2017large} differs from the $\beta$ in this paper by a power on $N$}. They prove that, in this regime, the tagged empirical field satisfies an LDP in $\mathcal{P}(\Sigma \times {\rm Config})$ (where $\Sigma$ denotes the support of the equilibrium measure) at speed $N$ with rate function
\begin{equation}
    \mathcal{F}( \overline{\mathbf{P}}) = \begin{cases}
       \mathcal{J}( \overline{\mathbf{P}})  - \inf_{\overline{\mathbf{P}}^* \in \mathcal{P}(\Sigma \times {\rm Config}) } \mathcal{J}( \overline{\mathbf{P}}^*) \ \ \ &\textrm{if}\ \ \  \overline{\mathbf{P}} \in \mathcal{P}_{s,1}(\Sigma \times {\rm Config})\\
       \infty \ \ \ \ \ \ \ \ \ \ \ \ \ \  \ \ \ \ \ \ \ \ \ \ \ \ &\textrm{if}\ \ \  \overline{\mathbf{P}} \notin \mathcal{P}_{s,1}(\Sigma \times {\rm Config})
    \end{cases} 
\end{equation}
with 
\begin{equation}
     \mathcal{J}( \overline{\mathbf{P}}) = \theta \overline{\mathbb{W}}( \overline{\mathbf{P}}, \mu_{V}) + \overline{\rm Ent}[ \overline{\mathbf{P}} | \overline{\mathbf{\Pi}}^{1}].
\end{equation}
The term $\overline{\mathbb{W}}$ is energy-derived and corresponds to the renormalized energy of point processes, see \cite{leble2017large} for an exact definition. The term $\overline{\rm Ent}$ is entropy derived, see Section \ref{sect:prelims} for an exact definition. A natural question is then to understand the behavior of the system in the limit as $\theta$ tends to $0$. We solve this question for general interactions in the regime $\beta = \frac{1}{N}$, and also for Riesz interactions in the regime $\frac{1}{N} \ll \beta \ll N^{\frac{2s - d}{d}} $. A natural expectation is that the rate function consists of dropping the energy term and keeping only the entropy term. This expectation turns out to be close to correct in the regime $\frac{1}{N}\ll \beta \ll N^{\frac{2s - d}{d}} $ for Riesz interactions, but completely wrong for the regime $\beta = \frac{1}{N}$ and general interactions. 

    \item[3.] There has been considerable interest in LDPs for particle systems. An LDP for the empirical measure (in various temperature regimes) has been derived both in the Coulomb setting \cite{arous1997large,petz1998logarithmic} and in the general setting \cite{chafai2014first,garcia2019large}. On the other hand, \cite{leble2017large,leble2017local,armstrong2021local} derive an LDP for the tagged empirical field in the {low temperature} regimes. However, an LDP for the tagged empirical field outside the {low temperature} regimes has been absent from the literature. This paper fills this gap. 
\end{itemize}

\begin{remark}
    Throughout the paper, we will commit the abuse of notation by not distinguishing between a measure and its density. 
\end{remark}

\section{Setting and main definitions}

This section will introduce the objects that the main theorem deals with. Most of the section discusses the tagged empirical field and related objects. 

We start with definitions related to energy.

\begin{definition}
We denote the self-interaction of a measure $\mu$ by $\mathcal{E}(\mu)$:
\begin{equation}
    \mathcal{E} (\mu) = \int_{\mathbb{R}^{d} \times \mathbb{R}^{d}} g(x-y) \, \mathrm d  \mu \otimes \mu(x,y).
\end{equation}

We denote the mean field limit of $\mathcal{H}_{N},$ by $\mathcal{E}_{V}:$
\begin{equation}\label{meanfieldlimit}
    \mathcal{E}_{V} (\mu) =  \mathcal{E} (\mu) + \int_{\mathbb{R}^{d}} V \, \mathrm d \mu. 
\end{equation}

We also introduce the thermal energy $ \mathcal{E}_{V}^{\theta}:$
\begin{equation}\label{thermallimit}
     \mathcal{E}_{V}^{\theta} (\mu) =  \mathcal{E}_{V} (\mu) + \frac{1}{ \theta} {\rm ent}[\mu], 
\end{equation}
with
\begin{equation}
{\rm ent}[\mu ]=
    \begin{cases}
    \int_{\mathbb{R}^{d}}  \log  ({ \mathrm d\mu})  \, \mathrm  d{\mu} \quad  \text{ if } \mu \text{ is absolutely continuous w.r.t. Lebesgue measure} \\
    \infty\quad  \text{ otherwise}
    \end{cases}
\end{equation} 
\end{definition}

\begin{definition}
We denote by $\mathcal{P}(\mathbb{R}^{d})$ (respectively $\mathcal{P}(\mathbb{Q}^{d})$) the set of probability measures on $\mathbb{R}^{d}$ (respectively on $\mathbb{Q}^{d}$ ).

We denote by $\mu_{\theta}$ the minimizer of $ \mathcal{E}_{V}^{\theta}$ in $\mathcal{P}(\mathbb{R}^{d})$:
\begin{equation}
\label{def:theqmeas}
    \mu_{\theta}:= {\rm argmin}_{\mu \in \mathcal{P}(\mathbb{R}^{d})}  \mathcal{E}_{V}^{\theta} (\mu).
\end{equation}
We will refer to $\mu_{\theta}$ as the thermal equilibrium measure, see section \ref{thermeqmeas} for existence, uniqueness, and basic properties.
\end{definition}

We now introduce a few basic operations and sets that will be relevant later on in the paper. 

\begin{definition}
Throughout this section and throughout the paper, we will use the notation $\theta_{\tau}$ for the translation by $\tau$:
\begin{equation}
 \theta_{\tau}(x) = x+\tau.   
\end{equation}

We will also use the notation $\square_{R}(x)$ for the square of side $R$ and center $x$:
\begin{equation}
    \square_{R}(x) = \left[ x - \frac{R}{2}, x + \frac{R}{2} \right]^{d}.
\end{equation}
We will use the notation
\begin{equation}
    \square_{R} =  \square_{R}(0).
\end{equation}

\end{definition}

We now introduce the main object that this paper deals with: the tagged empirical field. 

\begin{definition}\label{def:tagempfiel}
Given a bounded measurable set $\Omega \subset \mathbb{R}^{d},$ and $X_{N} \in \mathbb{R}^{d\times N}$, we define the \emph{tagged empirical field}
\begin{equation}
\label{def:eqPN}
     \overline{\mathbf{P}}_{N}(X_{N}) = \frac{1}{|\Omega|} \int_{\Omega} \delta_{\left(x, \theta_{N^{\frac{1}{d}}x} \cdot X_{N}'\right)} \, \mathrm d x,
\end{equation}
where $X_{N}' = N^{\frac{1}{d}} X_{N}$, $\theta_{\tau}$ denotes the element-wise translation by $\tau$, and $|\Omega|$ denotes the Lebesgue measure of $\Omega$. 
\end{definition}

Having defined the tagged empirical field, we identify some topological spaces of interest and state a few foundational results. This is needed in order to identify the topology of the LDP. 

\begin{definition}
\label{def:distconfig}
Given an open set $A \subset \mathbb{R}^{d}$, we define ${\rm Config}(A)$ to be the set of locally finite points configurations on $A$. Equivalently, ${\rm Config}(A)$ can be thought of as the set of non-negative, purely atomic Radon measures on $A$ giving an integer mass to singletons. Given a measurable set $B \subset A$, and $\mathcal{C} \in {\rm Config}(A)$ we denote by $|\mathcal{C}|(B)$ the number of points of $\mathcal{C}$ in $B$. 

We denote ${\rm Config}:= {\rm Config}(\mathbb{R}^{d})$. 

We endow ${\rm Config}$ with the topology induced by the topology
of weak convergence of Radon measures.

We define the following distance on ${\rm Config}$:
\begin{equation}
\label{def:dist}
    d_{\rm Config}(\mathcal{C}_{1}, \mathcal{C}_{2}) = \sum_{k=1}^{\infty} \frac{1}{2^{k}} \sup_{f \in {\rm Lip}_{1}(\mathbb{R}^{d})} \frac{\int_{\square_{K}} f \, \mathrm d \left(  \mathcal{C}_{1} - \mathcal{C}_{2} \right)}{|\mathcal{C}_{1}|(\square_{k}) + |\mathcal{C}_{2}|(\square_{k})},
\end{equation}
where ${\rm Lip}_{1}(\mathbb{R}^{d})$ denotes the set of Lipschitz functions on $\mathbb{R}^{d}$ with Lipschitz constant $1$ and such that ${\rm sup}\{|f|\} \leq 1$.
\end{definition}

The following Lemma establishes some basic properties about ${\rm Config}$.

\begin{lemma}
\begin{itemize}
    \item The topological space ${\rm Config}$ is Polish.
    \item The distance $ d_{\rm Config}$ is compatible with the topology induced by the topology
of weak convergence of Radon measures.
\end{itemize}
\end{lemma}

\begin{proof}
See for example \cite{leble2017large}, Lemma 2.1. 
\end{proof}

\begin{definition}
\label{def:disttagged}
Given a set $\Omega \subset \mathbb{R}^{d}$, we denote by $\mathcal{P}({\rm Config})$ the set of probability measures on ${\rm Config}$, and by $\mathcal{P}(\Omega \times {\rm Config})$ the set of probability measures on $\Omega \times {\rm Config}$. An element of $\mathcal{P}(\Omega \times {\rm Config})$ will be called a tagged point process. We endow $\mathcal{P}(\Omega \times {\rm Config})$ with the distance
\begin{equation}
\label{def:dist2}
    d_{\mathcal{P}(\Omega \times {\rm Config})} ( \overline{\mathbf{P}}_{1}, \overline{\mathbf{P}}_{2}) := \sup_{F \in {\rm Lip}_{1}(\Omega \times {\rm Config})} \int F \, \mathrm d ( \overline{\mathbf{P}}_{1} - \overline{\mathbf{P}}_{2}),
\end{equation}
where ${\rm Lip}_{1}(\Omega \times {\rm Config})$ denotes the space of Lipschitz functions on the space $(\Omega \times {\rm Config})$ with Lipschitz constant $1$ and such that $\mathrm{sup} |F| \leq 1$. The distance \eqref{def:dist2} metrizes the topology of weak convergence on $\mathcal{P}(\Omega \times {\rm Config})$ (see \cite{leble2017large}). 
\end{definition}

Now that we have defined the object that the main theorem deals with, and that we have defined a topology on the subspace in which it lives, we turn to define important quantities associated with tagged empirical fields. These are necessary for defining the rate function, and the space on which the LDP is proved. 

\begin{definition}
\label{def:disintegration}
Given $ \overline{\mathbf{P}} \in \mathcal{P}(\Omega \times {\rm Config})$, we define the disintegration measure $ \overline{\mathbf{P}}^{x}$, which for each $x \in \Omega$ is an element of $\mathcal{P}({\rm Config})$, characterized by the requirement that for any $F \in C^{0}(\Omega \times {\rm Config})$ we have
\begin{equation}
    \mathbf{E}_{ \overline{\mathbf{P}}}[F] = \frac{1}{|\Omega|} \int_{\Omega}  \mathbf{E}_{ \overline{\mathbf{P}}^{x}}[F(x, \cdot)] \, \mathrm d x.
\end{equation}

See \cite{ambrosio2005gradient} for the existence and uniqueness of the disintegration measure. 
\end{definition}

\begin{definition}
Given an open subset $A \subset \mathbb{R}^{d}$ and $\lambda > 0$, we define the Poisson Point Process of intensity $\lambda$ on $A$ as $\mathbf{\Pi}^{\lambda} \in \mathcal{P}({\rm Config}(A))$ characterized by the requirement that for any Borel set $B \subset A$,
\begin{equation}
\label{eq:Poisson}
    \mathbf{\Pi}^{\lambda} \left\{ |C|(B) = n \right\} = \frac{(\lambda|B|)^{n}}{n!} \exp \left( -\lambda |B| \right). 
\end{equation}
\end{definition}

\begin{definition}
Given a set $\Omega \subset \mathbb{R}^{d}$ and a measurable function $\lambda: \Omega \to \mathbb{R}^{+}$, we define the tagged Poisson Point Process $\overline{\mathbf{\Pi}}^{\lambda} \in \mathcal{P}(\Omega \times {\rm Config})$ as the unique tagged point process such that the disintegration measure satisfies that for each $x \in \Omega$,
\begin{equation}
  (\overline{\mathbf{\Pi}}^{\lambda})^{x} =   \mathbf{\Pi}^{\lambda(x)}. 
\end{equation}
\end{definition}

\begin{definition}
A point process ${\mathbf{P}} \in \mathcal{P}({\rm Config})$ is called stationary if for any set $A \subset {\rm Config}$ and any vector $\tau \in \mathbb{R}^{d}$,
\begin{equation}
    \mathbf{P}(A) = \mathbf{P}(\theta_{\tau} A).
\end{equation}
 The set of stationary point processes is denoted by $\mathcal{P}_{s}({\rm Config})$.

A tagged point process $ \overline{\mathbf{P}} \in \mathcal{P}(\Omega \times {\rm Config})$ is called stationary if $ \overline{\mathbf{P}}^{x}$ is stationary for all $x \in \Omega$. The set of stationary tagged point processes whose first marginal is the Lebuesgue measure on $\Omega$ is denoted by $ \overline{\mathbf{P}} \in \mathcal{P}_{s}(\Omega \times {\rm Config})$.
\end{definition}

\begin{definition}
\label{def:intensity}
 We define the intensity of a point process ${\mathbf{P}} \in \mathcal{P}({\rm Config})$ as 
 \begin{equation}
    {\rm int}[\mathbf{P}] := \mathbf{E}_{\mathbf{P}}[|C|\square_{1}].
 \end{equation}
 
 Note that, according to this definition, the ``Poisson Point Process of intensity $\lambda$" has intensity $\lambda$. Note also that if $\mathbf{P}$ is stationary, then the intensity is equal to 
 \begin{equation}
    \frac{1}{k^{d}} \mathbf{E}_{\mathbf{P}}[|C|\square_{k}]
 \end{equation}
for any $k > 0$.  
\end{definition}

\begin{definition}
We say that a tagged point process $ \overline{\mathbf{P}} \in \mathcal{P}(\Omega \times {\rm Config})$ has intensity $1$ if 
\begin{equation}
    \int_{\Omega} {\rm int} [ \overline{\mathbf{P}}^{x}] \, \mathrm  d x =1.
\end{equation}

The set of stationary tagged point processes of intensity $1$ is denoted by $ \mathcal{P}_{s,1}(\Omega \times {\rm Config})$.
\end{definition}
   
\begin{definition}
Given two point processes on a compact set $\Omega$, $\mathbf{P}_{1}, \mathbf{P}_{2} \in \mathcal{P}({\rm Config}(\Omega))$, we define the relative entropy of $\mathbf{P}_{1}$ with respect to $\mathbf{P}_{2}$ as 
\begin{equation}
   {\rm Ent}[\mathbf{P}_{1}|\mathbf{P}_{2}] = \begin{cases}
      \int \frac{ \mathrm d \mathbf{P}_{1}}{ \mathrm d \mathbf{P}_{2}} \log \left( \frac{ \mathrm d \mathbf{P}_{1}}{ \mathrm d \mathbf{P}_{2}} \right) \, \mathrm d \mathbf{P}_{2} \ {\rm if }\  \mathbf{P}_{1} \ll \mathbf{P}_{2}\\
      \infty \ {\rm otherwise, }
    \end{cases}
\end{equation}
where $\frac{d \mathbf{P}_{1}}{d \mathbf{P}_{2}}$ is the Radon–Nikodym derivative.

Given a stationary point process $\mathbf{P} \in \mathcal{P}_{s}({\rm Config})$, and $\lambda >0$ we define the specific relative entropy of $\mathbf{P}_{1}$ with respect to $\mathbf{\Pi}^{\lambda}$ as 
\begin{equation}
\label{def:relent}
     {\rm Ent}[\mathbf{P}|\mathbf{\Pi}^{\lambda}]=\lim_{R \to \infty} \frac{1}{R^{d}}  {\rm Ent}[\mathbf{P}|_{\square_{R}}|\mathbf{\Pi}^{\lambda}|_{\square_{R}}],
\end{equation}
where $\mathbf{\Pi}^{\lambda}|_{\square_{R}}$ and $\mathbf{P}|_{\square_{R}}$ denote the restrictions of $\mathbf{\Pi}^{\lambda}$ and $\mathbf{P}$ to the set $\square_{R}$, respectively \footnote{`restriction' is a bit sloppy, more precisely, we mean `projections on the set of intersections with'.}.

Given a tagged point processes $ \overline{\mathbf{P}} \in  \mathcal{P}_{s}(\Omega \times {\rm Config})$, and a measurable function $\lambda: \Omega \to \mathbb{R}^{+}$, we define the specific relative entropy of $ \overline{\mathbf{P}}$ with respect to $\overline{\mathbf{\Pi}}^{\lambda}$ as
\begin{equation}
    \overline{\rm Ent}[ \overline{\mathbf{P}} | \overline{\mathbf{\Pi}}^{\lambda}] = \int_{\Omega} {\rm Ent}[ \overline{\mathbf{P}}^{x} | \mathbf{\Pi}^{\lambda(x)}] \, \mathrm  d x.
\end{equation}
\end{definition}
Note that the specific relative entropy with respect to $\overline{\mathbf{\Pi}}^{\lambda}$ is not a function of the disintegration measure alone, since it depends on the configuration of points well as their density. 

The following lemma is classical (see, for example, \cite{leble2017large}) and establishes some basic properties about the entropy functional.  
\begin{lemma}
For any $\lambda>0$, there holds:
\begin{itemize}
    \item The limit in equation \eqref{def:relent} exists if $ {\mathbf{P}}$ is stationary.
    
    \item The map $\mathbf{P} \mapsto  {\rm Ent}[\mathbf{P}|\mathbf{\Pi}^{\lambda}]$ is affine and lower semi-continuous on $ \mathcal{P}_{s}({\rm Config})$.
    
    \item The sub-level sets of $ {\rm Ent}[\cdot|\mathbf{\Pi}^{\lambda}]$ are compact in $ \mathcal{P}_{s}({\rm Config})$ (it is a good rate function).
    
    \item We have $ {\rm Ent}[\mathbf{P}|\mathbf{\Pi}^{\lambda}] \geq 0$, and it vanishes only if $\mathbf{P} = \mathbf{\Pi}^{\lambda}$.
\end{itemize}
\end{lemma}

Lastly, we recall the definitions of rate functions and of Large Deviations Principle.

\begin{definition}[Rate function] Let $X$ be a metric space (or a topological space). A rate function is an lower semi-continuous  function $I:X \to [0,\infty]$, it is called a good rate function if its sublevel sets are compact.
\end{definition}

\begin{definition}[LDP] Let $P_{N}$ be a sequence of Borel probability measures on $X$ and $a_{N}$ a sequence of positive reals such that $a_{N} \to \infty.$ Let $I$ be a good rate function on $X.$ The sequence $P_{N}$ is said to satisfy a Large Deviations Principle (LDP) at speed $a_{N}$ with (good) rate function $I$ if for every Borel set $E \subset X$ the following inequalities hold:
\begin{equation}
    - \inf_{E^{\mathrm{o}}} I \leq \liminf_{N \to \infty} \frac{1}{a_{N}} \log \left( P_{N}(E) \right) \leq \limsup_{N \to \infty} \frac{1}{a_{N}} \log \left( P_{N}(E) \right) \leq  - \inf_{\overline{E}} I,
\end{equation}
where $E^{\mathrm{o}}$ and $\overline{E}$ denote respectively the interior and the closure of a set $E.$
Formally, this means that $P_{N}(E) \simeq \exp(-a_{N} \inf_{{E}} I).$
\end{definition}

We have now defined all objects needed for the statement of the main theorem, which we now introduce. 

\section{Main results}

The main result of this paper is the following theorem:

\begin{theorem}\label{LDPbeta1N}
Assume that $\beta = \frac{\theta}{N}$ for fixed $\theta>0$, and that $g: \mathbb{R}^{d} \to \mathbb{R}$ satisfies:
\begin{itemize}
    \item[1.] Symmetry.
    \begin{equation}
        g(x) = g(-x).
    \end{equation}
    
    \item[2.] Integrability.
    \begin{equation}
        g \in L^{1}_{\loc}(\mathbb{R}^{d}).
    \end{equation}
    
    \item[3.] Uniform continuity. For any $\epsilon >0,$ we have that $g$ is uniformly continuous on $\mathbb{R}^{d} \setminus B(0,\epsilon)$.
    
    \item[4.] Weak positive definiteness. If $\mu \in {BV}(\mathbb{R}^{d})$ is such that
    \begin{equation}
        \int_{\mathbb{R}^{d}} \, \mathrm d \mu =0,
    \end{equation}
    then
    \begin{equation}
        \mathcal{E}(\mu) \geq 0,
    \end{equation}
    where ${BV}(\mathbb{R}^{d})$ denotes the space of signed measures of bounded variation. 
    
    \item[5.]  For any $\mu \in \mathcal{P}(\mathbb{R}^{d})$,
    \begin{equation}
        \mathcal{E}(\mu) > - \infty.
    \end{equation}

    {\item[6.] There exists a function $D \in L^{1}_{\loc}(\mathbb{R}^{d})$ such that
    \begin{itemize}
        \item[6.1] $|g(x)| \leq D (x)$ for a.e. $x \in \mathbb{Q}^{d}$. 
        \item[6.2] $D$ is radially symmetric.
        \item[6.3] $D(r)$ is non-increasing.
    \end{itemize}} 
\end{itemize}

Assume that the confining potential $V$ satisfies:
\begin{itemize}
    \item[1.] $V$ is lower-semi-continuous.
    
    \item[2.] $V \in L^{1}(\mathbb{Q}^{d})$.

    \item[3.] $V$ is infinity outside of $\mathbb{Q}^{d}$.
\end{itemize}

Define $ \overline{\mathbf{P}}_{N}$ by Definition~\ref{def:tagempfiel} with $\Omega = \mathbb{Q}^{d}$ and $\mu_{\theta}$ by \eqref{def:theqmeas}. Then the push-forward of the Gibbs measure $\mathbf{P}_{N, \beta}$ (equation \eqref{Gibbs}) by $ \overline{\mathbf{P}}_{N}$ satisfies an LDP in $ \mathcal{P}(\mathbb{Q}^{d} \times {\rm Config})$ at speed $N$ with rate function 
\begin{equation}
    \mathcal{F}( \overline{\mathbf{P}}) =
    \begin{cases}
        \theta \mathcal{E}(\rho - \mu_{\theta}) + \overline{{\rm Ent}}[ \overline{\mathbf{P}} | \overline{\mathbf{\Pi}}^{\mu_{\theta}}] \ \ \ \textrm{if}\ \ \  \overline{\mathbf{P}} \in \mathcal{P}_{s,1}(\mathbb{Q}^{d} \times {\rm Config}),\\
        \infty \ \ \ \ \ \ \ \ \ \ \ \ \ \ \ \ \ \ \ \ \ \ \ \ \ \ \ \ \ \ \ \textrm{if}\ \ \  \overline{\mathbf{P}} \notin \mathcal{P}_{s,1}(\mathbb{Q}^{d} \times {\rm Config}),
    \end{cases}
\end{equation}
where $\rho$ is the probability measure with density ${\rm int}[ \overline{\mathbf{P}}^{x}]$. 
\end{theorem}

\begin{remark}
    
As mentioned in the introduction, the hypothesis that $V$ is infinity outside of $\mathbb{Q}^{d}$ is added so that we can effectively work on the hypercube $\mathbb{Q}^{d}$. Note that The hypothesis that $V \in L^{1}(\mathbb{Q}^{d})$ implies that $V$ is finite a.e. in $\mathbb{Q}^{d}$. It is straightforward to generalize our results to potentials $V$ that take the value $+\infty$ in $\mathbb{Q}^{d}$. 
\end{remark}

\begin{remark}
    An equivalent formulation of condition $6$ is the following: 
    \begin{itemize}
        \item[6.] The function $D: \mathbb{R}^{d} \to \mathbb{R}$, defined as 
        \begin{equation}
            D(x):=
            \begin{cases}
                \max_{ \left\{ s : |x| \leq |s| \leq \sqrt{d} T \right\}} |f(s)| &\textbf{ if }  |x| \leq \sqrt{d} T\\
                 \max_{\left\{ s :  |s| = \sqrt{d} T \right\}} |f(s)|&\textbf{ if } |x| \geq \sqrt{d} T,
            \end{cases}
        \end{equation}
        satisfies that $D \in L^{1}_{\loc}(\mathbb{R}^{d})$ \footnote{The $\sqrt{d} T$ factor comes from the diagonal of the hypercube in $d$ dimenions}. 
    \end{itemize}
\end{remark}

The way to prove Theorem \ref{LDPbeta1N} will be to first consider the case of non-interacting particles, and then generalize the statement to interacting particles. The statement for non-interacting particles is the following proposition, which is a new technical ingredient in this paper.

\begin{proposition}[LDP for non-interacting particles at high temperature]
\label{LDPnoninteracting}
Assume that $g=0$, $\beta = \frac{\theta}{N}$ for fixed $\theta>0$, and $V$ satisfies items $1-3$ of Theorem \ref{LDPbeta1N}. Define $\mu_{\theta}$ by \eqref{def:theqmeas} and define $ \overline{\mathbf{P}}_{N}$ by Definition~\ref{def:tagempfiel} with $\Omega = \mathbb{Q}^{d}$. Then the push-forward of $\mathbf{P}_{N, \beta}$ (equation \eqref{Gibbs}) by $ \overline{\mathbf{P}}_{N}$ satisfies an $LDP$ in $ \mathcal{P}(\mathbb{Q}^{d} \times {\rm Config})$ at speed $N$ and rate function
\begin{equation}
  \mathcal{F}( \overline{\mathbf{P}}) =\begin{cases}
      \overline{{\rm Ent}}[ \overline{\mathbf{P}} | \overline{\mathbf{\Pi}}^{\mu_{\theta}}] \ \ \ \ \ \ \ \ \ \ \ \ \ \ \ \ \ \ \, \ \textrm{if}\ \ \  \overline{\mathbf{P}} \in \mathcal{P}_{s,1}(\mathbb{Q}^{d} \times {\rm Config}),\\
        \infty \ \ \ \ \ \ \ \ \ \ \ \ \ \ \ \ \ \ \ \ \ \ \ \ \ \ \ \ \ \ \ \textrm{if}\ \ \  \overline{\mathbf{P}} \notin \mathcal{P}_{s,1}(\mathbb{Q}^{d} \times {\rm Config}),
  \end{cases} 
\end{equation}
\end{proposition}

\begin{remark}
It is easy to see that in the setting of Proposition \ref{LDPnoninteracting},
\begin{equation}
    \mu_{\theta} (x) = \frac{1}{z} \exp\left(-\theta{V(x)}\right),
\end{equation}
where
\begin{equation}
    z = \int_{\mathbb{Q}^{d}} \exp\left(-\theta{V(x)}\right) \, \mathrm  d x.
\end{equation}
We also have that 
\begin{equation}
    \mathrm d {\mathbf{P}}_{N,\beta} = \mu_{\theta}^{\otimes N} \mathrm d X_{N}
\end{equation}
and
\begin{equation}
    {Z}_{N, \beta} = z^{N}.
\end{equation}
\end{remark}

The strategy to prove Theorem \ref{LDPbeta1N} consists of 3 parts:
\begin{itemize}
    \item[1.] A splitting formula, that allows us to factorize the energy around the thermal equilibrium measure (limiting macroscopic density). 
    \item[2.] An LDP for a system of non-interacting particles (Proposition \ref{LDPnoninteracting}).  
    \item[3.] A study of the mean-field energy functional, that
allows us to derive the full LDP from the non-interacting case.
\end{itemize}

\cite{garcia2019large} uses a remarkably natural and general strategy to prove an LDP for the empirical measure in general interactions. However, when dealing with the tagged empirical field, a different approach is needed. The remarkable paper \cite{leble2017large} introduced several new technical ingredients, and this proof relies on them. However, it is still necessary to introduce new ingredients when dealing with a larger temperature regime and general interactions.  

\section{Literature review}

Rare events at the macroscopic scale were treated in \cite{garcia2019large} and \cite{chafai2014first} in the context of general interactions. The observable that allows us to analyze the macroscopic scale is the empirical measure (equation \eqref{empm}). In our setting, their main results are that the push-forward of the Gibbs measure (equation \eqref{Gibbs}) by the empirical measure (equation \eqref{empm}) satisfies an LDP in $\mathcal{P}(\mathbb{R}^{d})$ at speed $N^{2} \beta$ with a rate function given by
\begin{equation}
    \mathcal{E}_{V}(\cdot) - \min_{\mu } \mathcal{E}_{V}(\mu)
\end{equation}
if $\frac{1}{N} \ll \beta$, and 
\begin{equation}
    \mathcal{E}_{V}^{\theta}(\cdot) - \min_{\mu } \mathcal{E}_{V}^{\theta}(\mu)
\end{equation}
if $\beta = \frac{\theta}{N}$. The reference \cite{garcia2019large} is even more general since it treats interacting particle systems in general compact manifolds, and with an interaction that is given by a many body formula. This paper continues the investigation of \cite{garcia2019large} by analyzing rare events in the high-temperature regime for general interactions at the microscopic scale. 

The microscopic behavior of a general {interacting gas} in the high-temperature regime is also the subject of \cite{lambert2021poisson}. In this case, the author deals with the local point process
\begin{equation}
    \sum_{i=1}^{N} \delta_{N^{\frac{1}{d}} (x_{i} -x_{0} )},
\end{equation}
and shows that it converges to a Poisson Point Process, with density given by the thermal equilibrium measure at the point $x_{0}$. Even though the subject of this paper is also the microscopic behavior of a general {interacting gas} in the high-temperature regime, our results are to a large extent independent. Neither result implies the other, and the techniques used are quite different. Indeed, even though the LDP proved in Theorem \ref{LDPbeta1N} implies that the tagged empirical field converges to a tagged Poisson Point Process, it does not imply that $ \sum_{i=1}^{N} \delta_{N^{\frac{1}{d}} (x_{i} -x_{0} )}$ converges to a Poisson Point Process. Conversely, one cannot derive an LDP from the convergence result proved in \cite{lambert2021poisson}.     

As mentioned in the introduction, in \cite{leble2017large} the authors treat the Riesz and Coulomb cases at {low temperature} ($\beta = \theta N^{\frac{2s - d}{d}}$). The main result of \cite{leble2017large} was later extended to hyper-singular Riesz gases \cite{hardin2018large}, two-component plasmas \cite{leble2017largetwo}, and the local tagged empirical field of a one-component plasma \cite{leble2017local, armstrong2021local}. The main result in \cite{leble2017large} has a similar flavor to ours because the rate function involves the competition of two terms: one derived from the energy and one derived from the entropy. In contrast, in the case of a Riesz gas at an intermediate temperature regime, there is no competition between the terms: the energy imposes a constraint at the leading order, and the entropy appears at the next order (see Appendix). Unlike \cite{leble2017large}, the energy-derived term that appears in the rate function of Theorem \ref{LDPbeta1N} is not the renormalized energy, but rather a mean-field jellium-type energy. Indeed, it is not even clear what ``renormalized energy" means in the context of general  interactions. {The object of our LDP is basically the tagged empirical field as defined in \cite{leble2017large}. The main difference (in the definition of the observable) is that in our case, the domain of averaging is the (effective) entire space and not the support of the equilibrium measure. This is due to the fact that, unlike the equilibrium measure, the \emph{thermal} equilibrium measure does not have compact support and is everywhere positive (even in Euclidean space) if the potential $V$ is finite a.e. The definition of the tagged empirical field would be trivial if the domain of averaging were the whole Euclidean space. In the setting of the hypercube instead of the entire space, this obstacle is naturally eliminated since the domain of integration is always compact. Extending quantities to infinite space is a classical problem in statistical mechanics (see, for example, chapter 6 of \cite{friedli2017statistical} for a discussion of the ``energy density"), but the intrinsically long-range range nature of the interaction makes it non-trivial to adapt this general setting to our problem. A feature that our LDP has in common with \cite{leble2017large} is the presence of the tagged specific relative entropy in the rate function. Unlike \cite{leble2017large}, however, in our case, the entropy is taken with respect to an in-homogeneous Poisson Point Process. }

Apart from the microscopic and macroscopic scales, it is also possible to analyze rare events at a mesoscopic scale. In the Coulomb setting, this is the subject of \cite{padilla2024large}. In this case, the observable to analyze is the local empirical measure, defined as 
\begin{equation}
    \frac{1}{N^{1-\lambda d}} \sum_{i=1}^{N} \delta_{N^{\lambda}x_{i}}|_{\square_{R}},
\end{equation}
for $\lambda \in \left( 0, \frac{1}{d} \right)$. In this case, the typical event is that the local empirical measure approximates a uniform measure of density $\mu_{V}(0)$. The rare events are governed by an LDP in which the rate function contains either an entropy-derived term or an energy-derived term, depending on the magnitude of the temperature. 

\section{Preliminaries}
\label{sect:prelims}

Before starting the proof of Proposition \ref{LDPnoninteracting} and Theorem \ref{LDPbeta1N}, we state some general preliminary results, and introduce additional notation and definitions. 

\subsection{Additional notation and definitions}

We start by giving a few additional definitions and introducing additional notation. 

\begin{definition}

We introduce the notation 
\begin{equation}
    \begin{split}
        \mathcal{G}(\mu, \nu) &= \int_{\mathbb{R}^{d} \times \mathbb{R}^{d}} g(x-y)\, \mathrm d \mu \otimes \nu (x,y)\\
        \mathcal{G}^{\neq}(\mu, \nu) &= \int_{\mathbb{R}^{d} \times \mathbb{R}^{d} \setminus \Delta} g(x-y) \, \mathrm d \mu \otimes \nu (x,y),
    \end{split}
\end{equation}
where $\Delta = \{x,y \in \mathbb{R}^{d} \times \mathbb{R}^{d} : x=y \}$.

Given a measure $\mu$ on $\mathbb{R}^{d}$, we define
\begin{equation}
    h^{\mu} = g \ast \mu. 
\end{equation}

Given a measure $\mu$ on $\mathbb{R}^{d}$, and $X_{N} \in \mathbb{R}^{d \times N}$, we define
\begin{equation}
    {\rm F}_{N}(X_{N}, \mu) = \frac{1}{N^{2}} \sum_{i \neq j} g(x_{i} - x_{j}) + \mathcal{E}(\mu) - 2   \mathcal{G}({\rm emp}_{N}, \mu).
\end{equation}

\end{definition}

\begin{definition}
Given an open subset $A \subset \mathbb{R}^{d}$, and a positive measurable function $\mu: A \to \mathbb{R}^{+}$, we define the in-homogeneous Poisson Point Process of intensity $\mu$ on $A$ as $\mathbf{\Pi}^{\mu} \in \mathcal{P}({\rm Config}(A))$ characterized by the requirement that for any Borel set $B \subset A$,
\begin{equation}
    {\mathbf{\Pi}^{\mu} \left\{ |C|(B) = n \right\} = \frac{( \lambda m(B))^{n}}{n!} \exp \left( - \lambda m(B) \right),}
\end{equation}
where
\begin{equation}
    m(B) = \int_{B} \mu(x) \, \mathrm  d x.
\end{equation}
\end{definition}

\begin{definition}
Given two probability measures $\mu, \nu \in \mathcal{P}(\mathbb{R}^{d})$, we define the relative entropy of $\mu$ with respect to $\nu$ as 
\begin{equation}
    \begin{cases}
     {\rm ent}[\mu|\nu] = \int \frac{ \mathrm d \mu}{ \mathrm d \nu} \log \left( \frac{ \mathrm d \mu}{ \mathrm d \nu} \right) \, \mathrm d \nu \ {\rm if }\  \mu \ll \nu\\
     {\rm ent}[\mu|\nu] = \infty \ {\rm if }\ {\rm not,}
    \end{cases}
\end{equation}
where $\frac{ \mathrm d \mu}{ \mathrm d \nu}$ is the Radon–Nikodym derivative.
\end{definition}

\subsection{The thermal equilibrium measure}
\label{thermeqmeas}

In this section, we prove the existence, uniqueness, and some other basic properties of the thermal equilibrium measure. We also use the thermal equilibrium measure to derive a splitting formula for the energy. 

\begin{proposition}[Existence, uniqueness, characterization]

{Assume that $g$ satisfies items $2,4,5$ of the hypotheses of Theorem \ref{LDPbeta1N} and that $V$ satisfies items $1,2,3$ of the hypotheses of Theorem \ref{LDPbeta1N}.} 

Then the functional \eqref{thermallimit} has a unique minimizer in the set $\mathcal{P}(\mathbb{Q}^{d})$, which we denote $\mu_{\theta}.$ Additionally, $\mu_{\theta}$ satisfies that $\mu_{\theta} > 0$ a.e. in $\mathbb{Q}^{d}$ and the Euler-Lagrange equation

\begin{equation}
\label{eq:ELeq}
    2 h^{\mu_{\theta}} + V + \frac{1}{\theta} \log \mu_{\theta} = c,
\end{equation}
for some $c \in \mathbb{R}.$

\end{proposition}

\begin{proof}

\textbf{Step 1}[Existence and uniqueness]

Existence is a simple consequence of the direct method: Let $\mu_{n}$ be a sequence such that 
\begin{equation}
    \lim_{n \to \infty} \mathcal{E}^{\theta}_{V}(\mu_{n}) = \inf_{\mu \in \mathcal{P}(\mathbb{Q}^{d})} \mathcal{E}^{\theta}_{V}(\mu). 
\end{equation}
Then, modulo a subsequence, $\mu_{n}$ converges weakly to a measure $\mu$. By properties $1$ and $2$ of $V$, and $1$ and $2$ of $g$, we have that $\mathcal{E}_{V}$ is lower semi-continuous with respect to weak convergence and therefore
\begin{equation}
        \mathcal{E}_{V}(\mu) \leq \liminf_{n \to \infty} \mathcal{E}_{V}(\mu_{n}).
\end{equation}
It is well-known that the entropy functional is lower semi-continuous, and hence
\begin{equation}
    {\rm ent}[\mu] \leq \liminf_{n \to \infty} {\rm ent}[\mu_{n}].
\end{equation}
This implies that $\mu$ is a minimizer of $\mathcal{E}^{\theta}_{V}$ in $\mathcal{P}(\mathbb{Q}^{d})$. 

Uniqueness is a consequence of property $4$ of $g$: Proceed by contradiction and assume that two distinct minimizers $\mu_{1}, \mu_{2}$ exist. Then, since $\mathcal{E}$ is quadratic and by property $4$ of $g$,
\begin{equation}
    \begin{split}
         \mathcal{E}\left(\frac{\mu_{1}+\mu_{2}}{2}\right) &= \frac{1}{2}\left(  \mathcal{E}\left(\mu_{1}\right) +  \mathcal{E}\left(\mu_{2}\right) -2  \mathcal{E}\left( \frac{\mu_{1}-\mu_{2}}{2} \right) \right)\\
         &\leq\frac{1}{2}\left(  \mathcal{E}\left(\mu_{1}\right) +  \mathcal{E}\left(\mu_{2}\right) \right).
    \end{split}         
\end{equation}

 If $\mu_{1} \neq \mu_{2}$, by strict convexity of the function 
 \begin{equation}
     \mu \to {\rm ent}[\mu],
 \end{equation}
 there holds
 \begin{equation}
 \begin{split}
     \mathcal{E}^{\theta}_{V} \left( \frac{\mu_{1}+\mu_{2}}{2} \right) &< \frac{1}{2} \left( \mathcal{E}^{\theta}_{V}(\mu_{1}) + \mathcal{E}^{\theta}_{V}(\mu_{2}) \right)\\
     &= \inf_{\mu \in \mathcal{P}(\mathbb{Q}^{d})} \mathcal{E}^{\theta}_{V}(\mu).
\end{split}
\end{equation}
This is a contradiction and therefore the minimizer is unique. 

\textbf{Step 2}[Positivity]

The proof is standard, see for example \cite{armstrong2022thermal, neri2004statistical, rougerie2014quantum}.

We proceed by contrapositive, and assume that there is a bounded set $X \subset \mathbb{Q}^{d}$ such that $\mu_{\theta}(x)=0$ for all $x \in X$. Now consider
\begin{equation}
    {\mu}_{\theta}^{\epsilon} = \frac{\mu_{\theta} + \epsilon \mathbf{1}_{X}}{ 1 + \epsilon|X|}.
\end{equation}
Doing a Taylor expansion of $\mathcal{E}_{\theta}({\mu}_{\theta}^{\epsilon})$, we get that
\begin{equation}
    \mathcal{E}_{V}^{\theta}({\mu}_{\theta}^{\epsilon}) = \mathcal{E}_{V}^{\theta}(\mu_{\theta}) - \epsilon |X| \left( \mathcal{E}_{V}^{\theta}(\mu_{\theta}) \right)  + \epsilon \int_{X} h^{\mu_{\theta}} (x) + V(x) \, \mathrm d x +  \frac{1}{\theta}|X| \epsilon \log \epsilon + O(\epsilon^{2}). 
\end{equation}

Note that 
\begin{equation}
    \int_{X} h^{\mu_{\theta}} (x) + V(x) \, \mathrm d x < \infty,
\end{equation}
since $X$ is bounded, $g$ satisfies item $1$, $V$ satisfies item $2$, and $\int_{\mathbb{Q}^{d}} \mu_{\theta}\, \mathrm d x = 1$. We then get that 
\begin{equation}
     \mathcal{E}_{V}^{\theta}({\mu}_{\theta}^{\epsilon}) = \mathcal{E}_{V}^{\theta}(\mu_{\theta}) + \epsilon C +  \frac{1}{\theta}|X| \epsilon \log \epsilon + O(\epsilon^{2}),
\end{equation}
where $C$ depends on $X$.

If $|X| \neq 0$, this would imply that 
\begin{equation}
     \mathcal{E}_{V}^{\theta}({\mu}_{\theta}^{\epsilon}) < \mathcal{E}_{V}^{\theta}(\mu_{\theta}) 
\end{equation}
for $\epsilon >0$ small enough. 

This would be a contradiction and therefore $\mu_{\theta}$ is positive a.e. in $\mathbb{Q}^{d}$. 

\textbf{Step 3}[Euler-Lagrange equation]

Let $f$ be a smooth, compactly supported function such that $\int_{\mathbb{Q}^{d}} f \mu_{\theta} (x) \, \mathrm d x =0$. Note that $(1 + tf) \mu_{\theta}$ is a probability measure for small enough $|t|$. Since $\mu_{\theta}$ is a minimizer, we obtain that 
\begin{equation}
    \mathcal{E}_{V}^{\theta}(\mu_{\theta}) \leq  \mathcal{E}_{V}^{\theta}((1 + tf)\mu_{\theta}), 
\end{equation}
which implies, taking the derivative at $t=0$, that 
\begin{equation}
    \int_{\mathbb{Q}^{d}} (2h^{\mu_{\theta}} + V + \frac{1}{\theta} \log \mu_{\theta} ) f  \mu_{\theta} \, \mathrm d x =0.
\end{equation}

Since $\mu_{\theta} \neq 0$ a.e. we infer that 
\begin{equation}
    \int_{\mathbb{Q}^{d}} (2h^{\mu_{\theta}} + V + \frac{1}{\theta} \log \mu_{\theta} ) g \, \mathrm d x =0.
\end{equation}
for all $g$ such that $\int_{\mathbb{Q}^{d}} g(x) \, \mathrm d x =0$, which implies 
\begin{equation}
    2h^{\mu_{\theta}} + V + \frac{1}{\theta} \log \mu_{\theta} = c
\end{equation}
for some $c$.

\end{proof}

\begin{remark}
\label{rem:thembounded}
    Note that equation \eqref{eq:ELeq} may be rewritten as 
    \begin{equation}
    \label{eq:ELeq2}
        \mu_{\theta} = \exp \left( -\theta \left( V + h^{\mu_{\theta}} - c\right) \right). 
    \end{equation}
    Since $V$ is bounded below, and $h^{\mu} $ is bounded below uniformly for any $\mu \in \mathcal{P}(\mathbb{Q}^{d})$, we may infer that $\mu_{\theta} $ is bounded above, by a bound that depends on $\theta$. 

    Since $\mu_{\theta} $ is bounded, we may proceed as in the proof of Lemma \ref{lem:l.s.c.} and show that $ h^{\mu_{\theta}}$ is continuous. This, in turn, implies that $\mu_{\theta} $ is continuous by equation \eqref{eq:ELeq2}. Hence, $\mu_{\theta} $ is uniformly continuous $\mathbb{Q}^{d}$. 
\end{remark}

Now that we have proved the existence, uniqueness, and some basic properties of the thermal equilibrium measure; we will use it to derive a splitting formula for the Hamiltonian. This formula appeared in the Coulomb case in \cite{armstrong2021local}. 

\begin{proposition}[Thermal splitting formula]
Let $\theta >0$. We introduce the notation
\begin{equation}
    \zeta_{\theta} = - \frac{1}{\theta} \log(\mu_{\theta}).
\end{equation}

Then for any point configuration $X_{N}$ the Hamiltonian $\mathcal{H}_{N}$ can be rewritten (split) as  
\begin{equation}\label{eq:thermspltfrm}
    \mathcal{H}_{N} (X_{N}) = N^{2} \left(  \mathcal{E}_{V}^{\theta} (\mu_{\theta}) +  {\rm F}_{N}(X_{N}, \mu_{\theta})+ \int_{\mathbb{R}^{d}} \zeta_{\theta}\, \mathrm d {\rm emp}_{N} \right).
\end{equation}
\end{proposition}

\begin{proof}
It suffices to write 
\begin{equation}
    \begin{split}
        \mathcal{H}_{N}(X_{N}) &= N^{2} \left( \frac{1}{N^{2}} \sum_{i \neq j} g(x_{i} - x_{j}) + \int_{\mathbb{R}^{d}} V \, \mathrm d {\rm emp}_{N} \right) \\
        &= N^{2} \left(\mathcal{E} (\mu_{\theta}) + 2 \mathcal{G} (\mu_{\theta}, {\rm emp}_{N} - \mu_{\theta}) + {\rm F}_{N} ({\rm emp}_{N}, \mu_{\theta}) + \int_{\mathbb{R}^{d}} V \, \mathrm d {\rm emp}_{N} \right)
    \end{split}
\end{equation}
and then use the Euler-Langrange equation for $\mu_{\theta}$.
\end{proof}

\subsection{Next order partition function}

In analogy with previous work in the field \cite{armstrong2021local, leble2017large}, we define a next-order partition function; which will in practice be a negligible error term in the rest of the paper. 

\begin{definition}
We define the next order partition function $K_{N, \beta}$ as 
\begin{equation}
\label{eq:defnext}
    K_{N, \beta} = \frac{Z_{N, \beta}}{ \exp \left( -{ N \theta}  \mathcal{E}_{V}^{\theta}(\mu_{\theta}) \right)},
\end{equation}
with $\theta = N \beta$.
\end{definition}

\begin{proposition}
\label{limitofpartitionfunction}
Assume that $\beta = \frac{\theta}{N}$ for a fixed $\theta \in \mathbb{R}^{+}$. Then 
\begin{equation}
    \lim_{N \to \infty} \frac{\log \left( K_{N, \beta} \right)}{N}  = 0.
\end{equation}
\end{proposition}

\begin{proof}
The strategy of the proof will be to use the Laplace Principle proved in \cite{garcia2019large}.

Consider the probability measure $\pi,$ defined as 
\begin{equation}
    \mathrm d \pi (y)= \frac{1}{{z}^{*}} \exp \left( - V(y) \right) {\, \mathrm d y},
\end{equation}
where
\begin{equation}
    {z}^{*} = \int_{\mathbb{Q}^{d}} \exp \left( - V(x) \right) \, \mathrm d x.
\end{equation}

Consider also the Hamiltonian
\begin{equation}
    \mathcal{H}^{*}_{N}(X_{N}) = \sum_{i\neq j} g(x_{i}- x_{j}) ,
\end{equation}
with mean field limit
\begin{equation}
    \mathcal{E}^{*} (\mu) = \int_{\mathbb{Q}^{d} \times \mathbb{Q}^{d}} g(x-y) \mathrm d \mu(x) \mathrm d \mu(y). 
\end{equation}

Then by \cite{garcia2019large}, we have that the following Laplace principle holds: for every bounded and continuous function $f :  \mathcal{P}(\mathbb{Q}^{d}) \to \mathbb{R}$
\begin{equation}
    \lim_{N \to \infty} \frac{1}{N \theta} \log \left(  \int_{\mathbb{Q}^{d \times N}} \exp \left( - N \theta f ({\rm emp}_{N}) \right)\, \mathrm  d \gamma_{N} \right) = \inf_{\mu \in \mathcal{P}(\mathbb{Q}^{d})} \{ f(\mu) + F(\mu) \},
\end{equation}
where the probability measure $\gamma_{N}$ is defined as 
\begin{equation}
    \mathrm d \gamma_{N} = \exp \left( - \beta  \mathcal{H}^{*}_{N} \right)  \mathrm d \pi ^{\otimes N}, 
\end{equation}
and $F$ is defined as 
\begin{equation}
    F(\mu) =  \mathcal{E}^{*} (\mu) + \frac{1}{\theta} {\rm ent}[\mu | \pi].
\end{equation}

In particular, taking $f =0,$ we have that 
\begin{equation}
    \lim_{N \to \infty} \frac{1}{N \theta} \log \left(  \int_{\mathbb{Q}^{d \times N}} \exp \left( - \beta \mathcal{H}^{*}_{N} \right) \, \mathrm d \pi ^{\otimes N} \right) = \inf_{\mu \in \mathcal{P}(\mathbb{Q}^{d})} \{  \mathcal{E}^{*} (\mu) + \frac{1}{\theta} {\rm ent}[\mu | \pi] \}.
\end{equation}

Note that for any $\mu \in \mathcal{P}(\mathbb{Q}^{d})$ we have that
\begin{equation}
    \mathcal{E}^{*} (\mu) + \frac{1}{\theta} {\rm ent}[\mu | \pi] =  \mathcal{E}_{V} (\mu) +\frac{1}{\theta} \left( {\rm ent}[\mu] - \log z^{*} \right).
\end{equation}

On the other hand, for any $\mu \in \mathcal{P}(\mathbb{Q}^{d})$ we have that
\begin{equation}
    \frac{1}{N \theta} \log \left(  \int_{\mathbb{Q}^{d \times N}} \exp \left( - \beta  \mathcal{H}^{*}_{N} \right) \, \mathrm d \pi ^{\otimes N} \right) =  \frac{1}{N \theta} \log \left(  \int_{\mathbb{Q}^{d \times N}} \exp \left( - \beta  \mathcal{H}_{N} \right) \, \mathrm d X_{N} \right)  - \frac{ \log z^{*}}{ \theta}.
\end{equation}

Therefore 
\begin{equation}
    \lim_{N \to \infty} \frac{1}{N \theta} \log \left(  \int_{\mathbb{Q}^{d \times N}} \exp \left( - \beta \mathcal{H}_{N} \right) \, \mathrm d X_{N} \right) = \inf_{\mu \in \mathcal{P}(\mathbb{Q}^{d})} \{  \mathcal{E}_{V} (\mu) +\frac{1}{\theta} \left( {\rm ent}[\mu] \right) \},
\end{equation}
which implies that 
\begin{equation}
        \lim_{N \to \infty} \frac{1}{N \theta} \log \left(  Z_{N, \beta} \right) = \mathcal{E}_{V}^{\theta}(\mu_{\theta}),
\end{equation}
and therefore by equation \eqref{eq:defnext} that 
\begin{equation}
    \lim_{N \to \infty} \frac{\log \left( K_{N, \beta} \right)}{N}  = 0.
\end{equation}
\end{proof}

\begin{remark}
    It is possible to give a simpler proof of this result, without introducing $\pi$, $\mathcal{H}^{*}_{N}$, or $\mathcal{E}^{*}$. However, we present this proof because it can be generalized to the case in which $V$ is not necessarily infinite outside of $\mathbb{Q}^{d}$. 
\end{remark}

\subsection{Mean-field compatibility, compactness, and lower semi-continuity}
\label{sect:comp}

We now derive fundamental tools about the energy functional: mean-field compatibility, compactness, and lower semi-continuity.

\begin{lemma}[Mean field compatibility]
\label{l:mfcompatibility}
For any $V: \mathbb{R}^{d} \to \mathbb{R}$ that satisfies items $1-3$ of Theorem \ref{LDPbeta1N}, and any $g$ satisfying items $1-5$ of Theorem \ref{LDPbeta1N}, we have that
    \begin{equation}
    \label{eq:mfcompatibility}
        \min_{\mu \in \mathcal{P}(\mathbb{R}^{d})}   \mathcal{E}_{V} (\mu) = \lim_{N \to \infty} \min_{X_{N} \in \mathbb{Q}^{d \times N}} \left(  \frac{1}{N^{2}}  \mathcal{H}_{N}(X_{N}) \right),
    \end{equation}
    where $\mathcal{H}_{N}$ is given by \eqref{eq:hamilton} and $\mathcal{E}_{V}$ is given by \eqref{meanfieldlimit}. 
\end{lemma}

\begin{proof}
See \cite{borodachov2019discrete}, Theorem 4.2.2.
\end{proof}

\begin{remark}
We thank Ed Saff for introducing us to this result. 
\end{remark}

{\begin{lemma}
\label{lem:convden}
Let $X_{N} \in \mathbb{Q}^{d \times N}$, and let ${\rm emp}_{N}$ and $ \overline{\mathbf{P}}_{N}$ be as in equation \eqref{empm} and Definition \ref{def:tagempfiel} with $\Omega = \mathbb{Q}^{d}$, respectively. If
\begin{equation}
    \overline{\mathbf{P}}_{N}(X_{N}) \to \overline{\mathbf{P}},
\end{equation}
 for some $ \overline{\mathbf{P}} \in \mathcal{P}_{s,1}(\mathbb{Q}^{d} \times {\rm Config})$, then 
\begin{equation}
    {\rm emp}_{N} \rightharpoonup \rho,
\end{equation}
weakly in the sense of probability measures, where $\rho$ is defined as 
\begin{equation}
    \rho (x)  :=  {\rm int}[ \overline{\mathbf{P}}^{x}].
\end{equation}
\end{lemma}

\begin{proof}
Since $\mathbb{Q}^{d}$ is compact, modulo a subsequence (not relabelled),
\begin{equation}
    {\rm emp}_{N} \rightharpoonup \widetilde{\rho},
\end{equation}
weakly in the sense of probability measures for some $\widetilde{\rho} \in \mathcal{P}(\mathbb{Q}^{d})$. We now claim that $\widetilde{\rho} = \rho$. To prove this claim, note that for any measurable set $\Omega \subset \mathbb{Q}^{d} $
\begin{equation}
    \begin{split}
        \int_{\Omega} \mathbf{E}_{ \overline{\mathbf{P}}_{N}^{x}} [{\rm Num}(\square_{1})] \, \mathrm d x &=  \int_{\Omega} \left| \theta_{N^{\frac{1}{d}} x} X_{N}' \big|_{\square_{1}} \right| \, \mathrm d x \\
        &\geq \int_{\Omega} {\rm emp}_{N} \, \mathrm d x,
    \end{split}
\end{equation}
where ${\rm Num}(\square_{1})$ denotes the number of points of a point configuration in $\square_{1}$. 

Letting $N$ tend to $\infty$ and using the definition of intensity and weak convergence, we have that for any $\Omega \subset \mathbb{Q}^{d}$,
\begin{equation}
    \int_{\Omega} \widetilde{\rho}(x) \, \mathrm d x  \leq \int_{\Omega} {\rm int}[ \overline{\mathbf{P}}^{x}] \, \mathrm d x.
\end{equation}

On the other hand, by Definition \ref{def:disintegration} and Definition \ref{def:intensity},
\begin{equation}
\begin{split}
 \int_{\mathbb{Q}^{d}} {\rm int}[ \overline{\mathbf{P}}^{x}] \, \mathrm d x &=\lim_{N \to \infty} \int_{\mathbb{Q}^{d}} {\rm int}[ \overline{\mathbf{P}}_{N}(X_{N})^{x}] \, \mathrm d x \\
&= \lim_{N \to \infty} \int_{\mathbb{Q}^{d}}  \mathbf{E}_{ \overline{\mathbf{P}}_{N}(X_{N})^{x}}[|C|\square_{1}] \, \mathrm d x \\
&= \lim_{N \to \infty} \int_{\mathbb{Q}^{d}} [|\theta_{N^{\frac{1}{d}}x} \cdot X_{N}'|\square_{1}] \, \mathrm d x \\
&= \lim_{N \to \infty} \int_{\mathbb{Q}^{d}}\, \mathrm d {\rm emp}_{N} \\
&=  \int_{\mathbb{Q}^{d}} \widetilde{\rho}(x) \, \mathrm d x. 
\end{split}    
\end{equation}
Therefore 
\begin{equation}
    \int_{\mathbb{Q}^{d}} \widetilde{\rho}(x) \, \mathrm d x=1,  
\end{equation}
which implies that 
\begin{equation}
    \widetilde{\rho} (x) = {\rm int}[ \overline{\mathbf{P}}^{x}].
\end{equation}

Since every subsequence of ${\rm emp}_{N}$ has a further subsequence that converges to $\rho$, we conclude that ${\rm emp}_{N}$ itself converges to $\rho$.  
\end{proof}

\begin{lemma}
\label{lem:convden2}
Let $X_{N} \in \mathbb{Q}^{d \times N}$, and let ${\rm emp}_{N}$ and $ \overline{\mathbf{P}}_{N}$ be as in equation \eqref{empm} and Definition \ref{def:tagempfiel} with $\Omega = \mathbb{Q}^{d}$, respectively. Then a subsequence (not relabelled) satisfies that
\begin{equation}
    \overline{\mathbf{P}}_{N}(X_{N}) \to \overline{\mathbf{P}},
\end{equation}
for some $ \overline{\mathbf{P}} \in \mathcal{P}_{s,1}(\mathbb{Q}^{d} \times {\rm Config})$, and that
\begin{equation}
    {\rm emp}_{N} \rightharpoonup \rho,
\end{equation}
weakly in the sense of probability measures, where 
\begin{equation}
    \rho := {\rm int}[ \overline{\mathbf{P}}^{x}].
\end{equation}
\end{lemma}

\begin{proof}
Since $\mathbb{Q}^{d}$ is a compact space, by Proposition 3.5 of \cite{hardin2018large},
\begin{equation}
    \overline{\mathbf{P}}_{N}(X_{N}) \to \overline{\mathbf{P}},
\end{equation}
for some $ \overline{\mathbf{P}} \in \mathcal{P}_{s}(\mathbb{Q}^{d} \times {\rm Config})$. Proposition 3.5 of \cite{hardin2018large} actually deals with a Riesz interaction on Euclidean space. However, the proof of this specific result does not rely on the interaction or the space. We quote the proof from Proposition 3.5 of \cite{hardin2018large}: ``It is not hard to check that $\{\overline{P}_{N}\}_{N}$ converges (up to extraction) to some $\overline{P}$ in $\overline{\mathcal{M}}(\overline{\mathcal{X}})$ (indeed, the average number of points per unit volume is constant, which implies tightness, see, e.g., [18, Lemma 4.1]) whose stationarity is clear (see again, e.g., [18])." In our notation, $\{\overline{P}_{N}\}_{N}$ corresponds to $ \overline{\mathbf{P}}_{N}(X_{N})$, and the set $\overline{\mathcal{M}}(\overline{\mathcal{X}})$ corresponds to $\mathcal{P}_{s}(\mathbb{Q}^{d} \times {\rm Config})$. Reference [18] in the reference list of \cite{hardin2018large} is \cite{leble2017large} in our reference list. 

Lemma \ref{lem:convden} then implies that,
\begin{equation}
    {\rm emp}_{N} \rightharpoonup \rho,
\end{equation}
weakly in the sense of probability measures, which implies that $ \overline{\mathbf{P}} \in \mathcal{P}_{s,1}(\mathbb{Q}^{d} \times {\rm Config})$.
\end{proof}}

Having proved compactness, we now turn to prove lower semi-continuity.

\begin{lemma}\label{lem:l.s.c.}
Let $X_{N} \in \mathbb{R}^{d}$ be such that
\begin{equation}
     {\rm emp}_{N} \rightharpoonup \mu
\end{equation}
weakly in the sense of probability measures for some $\mu \in \mathcal{P}(\mathbb{Q}^{d})$. Then, if either $\nu \in \mathcal{P}(\mathbb{Q}^{d}) \cap L^{\infty}(\mathbb{Q}^{d}) $, or $\mu = \nu$, we have that
\begin{equation}
\label{eq:lsc}
    \mathcal{E}(\mu - \nu) \leq \liminf_{N \to \infty}  {\rm F}_{N}({\rm emp}_{N} , \nu). 
\end{equation}
\end{lemma}

\begin{proof}
\textbf{Step 1}[Case $\mu = \nu$]. 

First, we claim that $h^{\nu} $ is lower semicontinuous. To prove this claim, let $x_n \to x \in \mathbb{Q}^{d}$. Note that the sequence of functions $f_n (y):= g(x_n - y) \nu(y)$ converges pointwise to $f (y):= g(x - y) \nu(y)$. Hence, by Fatou's lemma, 
\begin{equation}
    \begin{split}
        h^{\nu}(x) &= \int_{\mathbb{Q}^{d}} f_n (y) \, \mathrm d y \\
        &\leq \liminf \int_{\mathbb{Q}^{d}} f_n (y) \, \mathrm d y \\
        &\leq \liminf  h^{\nu}(x_n).  
    \end{split}
\end{equation}
Note also that $h^{\nu} \in L^{1}(\mathbb{Q}^{d})$ if $\mu \in \mathcal{P}(\mathbb{Q}^{d})$. 

Using mean field compatibility (Lemma \ref{l:mfcompatibility}) with $V(x) = h^{\nu}(x) + \mathcal{E}(\nu) + \infty \mathbf{1}_{\mathbb{R}^{d} \setminus \mathbb{Q}^{d}}$, we then have that
\begin{equation}
    \begin{split}
        \liminf_{N \to \infty} {\rm F}_{N}({\rm emp}_{N}(X_{N}), \nu) &\geq \liminf_{N \to \infty} \min_{Y_{N} \in \mathbb{Q}^{d \times N}} {\rm F}_{N}({\rm emp}_{N}(Y_{N}), \nu) \\
        &\geq \min_{\rho \in \mathcal{P}(\mathbb{Q}^{d})} \mathcal{E}(\rho - \nu)\\
        &= 0 \\
        &= \mathcal{E}(\mu - \nu).
    \end{split}
\end{equation}

\textbf{Step 2}[General case]

For the general case, we write
\begin{equation}
    {\rm F}_{N}({\rm emp}_{N}(X_{N}), \nu) =  {\rm F}_{N}({\rm emp}_{N}(X_{N}) , \mu) + 2 \mathcal{G} ({\rm emp}_{N}(X_{N}) - \mu, \mu - \nu) + \mathcal{E} (\mu - \nu). 
\end{equation}

By Step 1, we have that 
\begin{equation}
   \liminf {\rm F}_{N}({\rm emp}_{N}(X_{N}) , \mu) \geq 0.
\end{equation}

We now claim that. 
\begin{equation}
    \liminf \mathcal{G} ({\rm emp}_{N}(X_{N}) - \mu, \mu - \nu) \geq 0.
\end{equation}

To prove this claim, note that $h^{\mu}$ is bounded below. Since $h^{\mu}$ is also lower semicontinuous, we have by definition of weak convergence that 
\begin{equation}
    \int_{\mathbb{Q}^{d}} \mu (x)h^{\mu}(x) \, \mathrm d x \leq \liminf  \int_{\mathbb{Q}^{d}} h^{\mu}(x) \, \mathrm d {\rm emp}_{N}. 
\end{equation}

Note also that $h^{\nu} \in L^{\infty}$. We claim that $h^{\nu}$ is also continuous. To prove this claim, let $x_n \to x \in \mathbb{Q}^{d}$. Note that we may write 
\begin{equation}
        h^{\nu} (x) -  h^{\nu} (x_n) =  \int_{\mathbb{Q}^{d}} (g(x_n - y) -g(x - y))\nu(y)  \, \mathrm d y .
\end{equation}

Hence, for any $\delta >0$, 
\begin{equation}
\begin{split}
        &\left|  h^{\nu} (x) -  h^{\nu} (x_n)  \right| \\
        \leq & \left|  \int_{\mathbb{Q}^{d} B(x, \delta)} (g(x_n - y) -g(x - y))\nu(y)  \, \mathrm d y  \right|+  \left|  \int_{  B(x, \delta)} (g(x_n - y) -g(x - y))\nu(y)  \, \mathrm d y  \right|.
\end{split}
\end{equation}

Note that 
\begin{equation}
    \lim_{n \to \infty} \left|  \int_{\mathbb{Q}^{d} \setminus B(x, \delta)} (g(x_n - y) -g(x - y))\nu(y)  \, \mathrm d y  \right|=0,
\end{equation}
since $g$ is uniformly continuous on $\mathbb{Q}^{d} \setminus B(0, \delta)$ by property $3$ of $g$ (see Theorem \ref{LDPbeta1N}). 

On the other hand, 
\begin{equation}
\label{eq:1}
    \begin{split}
        &\limsup_{n \to \infty}  \left|  \int_{ B(x, \delta)} (g(x_n - y) -g(x - y))\nu(y)  \, \mathrm d y  \right|\\
        \leq & 2 \limsup_{n \to \infty} \| \nu \|_{L^{\infty}} \int_{ B(0, \delta + |x - x_n|)} \left| g(y) \right| \, \mathrm d y  \\
        \leq & 2 \| \nu \|_{L^{\infty}} \int_{ B(0, \delta)} \left| g(y) \right| \, \mathrm d y.  
    \end{split}
\end{equation}
Since equation \eqref{eq:1} is valid for any $\delta >0$, and 
\begin{equation}
    \lim_{\delta \to 0} \int_{ B(0, \delta)} \left| g(y) \right| \, \mathrm d y,
\end{equation}
we may conclude that 
\begin{equation}
   \lim_{n \to \infty}  \left|  h^{\nu} (x) -  h^{\nu} (x_n)  \right| =0,
\end{equation}
and hence that $h^{\nu}$ is continuous.  

Since $h^{\nu}$ is continuous and bounded, we have by definition of weak convergence that 
\begin{equation}
    \int_{\mathbb{Q}^{d}} \mu (x)h^{\nu}(x) \, \mathrm d x = \lim  \int_{\mathbb{Q}^{d}} h^{\nu}(x) \, \mathrm d {\rm emp}_{N}. 
\end{equation}

From this, we may conclude equation \eqref{eq:lsc}. 
\end{proof}

\section{Proof of Proposition \ref{LDPnoninteracting}}

We now turn to prove Proposition \ref{LDPnoninteracting}, which is a particular case of Theorem \ref{LDPbeta1N}. This proposition will be a necessary step in the proof of Theorem \ref{LDPbeta1N}. 

The proof will require a previous result from \cite{leble2017large} that allows us to focus only on the local behaviour of a point process, by approximating an arbitrary point configuration, by its restriction to a compact set.
\begin{lemma}\label{lem:lebser1}
For any $\delta>0$ there exists an $R>0$ such that 
\begin{equation}
\label{eq:localizationpoints}
    \sup_{F \in {\rm Lip}_{1}({\rm Config})} \sup_{C \in {\rm Config}} \left| F(C) - F(C \cap \square_{R}) \right| < \delta.
\end{equation}
\end{lemma}

\begin{proof}
See \cite{leble2017large}, Lemma 2.1.
\end{proof}

It is clear that, given $C \in {\rm Config}$, we can find $R$ satisfying equation \eqref{eq:localizationpoints}. The point of Lemma \ref{lem:lebser1} is that is shows that there exists an $R$ such that equation \eqref{eq:localizationpoints} holds \emph{for all} $C \in {\rm Config}$.

We now embark on the proof of Proposition \ref{LDPnoninteracting}, which we restate here for convenience. 

\begin{proposition}[LDP for non-interacting particles at high temperature]
Assume that $g=0$, $\beta = \frac{\theta}{N}$ for fixed $\theta>0$, and $V$ satisfies items $1-3$ of Theorem \ref{LDPbeta1N}. Define $\mu_{\theta}$ by \eqref{def:theqmeas} and define $ \overline{\mathbf{P}}_{N}$ by Definition~\ref{def:tagempfiel} with $\Omega = \mathbb{Q}^{d}$. Then the push-forward of $\mathbf{P}_{N, \beta}$ (equation \eqref{Gibbs}) by $ \overline{\mathbf{P}}_{N}$ satisfies an $LDP$ in $ \mathcal{P}(\mathbb{Q}^{d} \times {\rm Config})$ at speed $N$ and rate function
\begin{equation}
  \mathcal{F}( \overline{\mathbf{P}}) =\begin{cases}
      \overline{{\rm Ent}}[ \overline{\mathbf{P}} | \overline{\mathbf{\Pi}}^{\mu_{\theta}}] \ \ \ \ \ \ \ \ \ \ \ \ \ \ \ \ \ \ \, \ \textrm{if}\ \ \  \overline{\mathbf{P}} \in \mathcal{P}_{s,1}(\mathbb{Q}^{d} \times {\rm Config}),\\
        \infty \ \ \ \ \ \ \ \ \ \ \ \ \ \ \ \ \ \ \ \ \ \ \ \ \ \ \ \ \ \ \ \textrm{if}\ \ \  \overline{\mathbf{P}} \notin \mathcal{P}_{s,1}(\mathbb{Q}^{d} \times {\rm Config}),
  \end{cases} 
\end{equation}
\end{proposition}

The proof will consist of three steps. The first one is an LDP for the empirical field on an inhomogeneous Poisson Point Process. The second one is an LDP for the \emph{tagged} empirical field of an inhomogeneous Poisson Point Process. The third and final one is an LDP for Gibbs measure of a Hamiltoninan with no repulsive interaction, instead of an inhomogeneous Poisson Point Process.

\begin{proof}

{
\textbf{Step 1}[LDP for the empirical field of a Poisson Point Process with variable intensity]

Let $\Lambda_{N} = N^{\frac{1}{d}} \mathbb{Q}^{d}$ and let $\phi$ be a positive continuous function on $\mathbb{Q}^{d}$. Let $\mathbf{\Pi}_{N}$ be a Poisson Point Process on $\Lambda_{N}$ with intensity $\phi (N^{-\frac{1}{d}} x )$. Let ${\mathbf{F}}_{N}^{\phi}$ be defined as 
\begin{equation}
    {\mathbf{F}}_{N}^{\phi}(C) = \frac{1}{ T^{d}} \int_{\mathbb{Q}^{d}} \delta_{\alpha_{\phi(x)} \left(\theta_{N^{\frac{1}{d}}x} \cdot C \right)} \, \mathrm d x,
\end{equation}
 where $\alpha_{\lambda}: {\rm Config} \to {\rm Config}$ is a dilation by $\lambda$. Then we claim that the push-forward of $\mathbf{\Pi}_{N}$ by ${\mathbf{F}}_{N}$ satisfies an LDP in $ \mathcal{P}({\rm Config})$ at speed $N$ with rate function
\begin{equation}
    \mathcal{F}({\mathbf{P}}) = 
         {\rm Ent}[{\mathbf{P}} | {\mathbf{\Pi}}^{1}].  
\end{equation}

In the case $\phi=1$, this is \cite[Proposition 7.5]{leble2017large} and it follows from \cite[Theorem 3.1]{georgii1993large}, along with \cite[Remark 2.4]{georgii1993large} to get rid of the periodization in the definition of the empirical field. In the case of a general $\phi,$ let $S>0$. Then the distribution of $\alpha_{\phi(x)} \left(\theta_{N^{\frac{1}{d}}x} \cdot C\right) \Big|_{\square_S}$ is a Poisson Point Process with density uniformly close to $1$, since $\phi$ is uniformly continuous, and so rare events are similar. We make this intuition rigorous by proving exponential equivalence. 

\textbf{Substep 1.1}[Exponential equivalence]

We claim that the push-forward of $\mathbf{\Pi}_{N}$ by 
\begin{equation}
     {\mathbf{F}}_{N}^{^{\phi},S}(C) := \frac{1}{ T^{d}} \int_{\mathbb{Q}^{d}} \delta_{\alpha_{\phi(x)} \left(\theta_{N^{\frac{1}{d}}x} \cdot C \right)\Big|_{\square_S}} \, \mathrm d x,
\end{equation}
and the push-forward of $\mathbf{\Pi^{1}}$ by 
\begin{equation}
     {\mathbf{F}}_{N}(C) := \frac{1}{ T^{d}} \int_{\mathbb{Q}^{d}} \delta_{ \left(\theta_{N^{\frac{1}{d}}x} \cdot C \right)\Big|_{\square_S} } \, \mathrm d x,
\end{equation}
are exponentially equivalent.

To prove this claim, divide $\mathbb{Q}^{d}$ into hypercubes of size $S N^{-\frac{1}{d}}$, denoted $Q^{S}_{i}$, with centers $x^{s}_{i}$. Arguing as in the proof of \cite[Lemma 6.5]{leble2017large}, it suffices to prove that the discrete approximation, the push-forward of $\mathbf{\Pi}_{N}$ by
\begin{equation}
     {\mathbf{F}}_{N, \rm disc }^{\phi, S}(C) := \left \lfloor{\frac{ S}{T N^{\frac{1}{d}}}}\right \rfloor^{d}  \sum_{i} \delta_{\alpha_{\phi(x_{i})} \left(\theta_{N^{\frac{1}{d}}x_{i}} \cdot C\right)\Big|_{\square_S} } 
\end{equation}
is exponentially equivalent to the push-forward of $\mathbf{\Pi^{1}}$ by ${\mathbf{F}}_{N}$. 

To prove this, define the function $\phi_S$ as 
\begin{equation}
    \phi_S (x) = \frac{\phi(N^{-\frac{1}{d}}x)}{\phi(N^{-\frac{1}{d}}x_{i})}, \ \ \ \ \ x \in N^{\frac{1}{d}}Q^{S}_{i}.
\end{equation}
consider the Poisson Point Process of intensity $1$ on $\Lambda_{N} \times (0,\infty)$, denoted $\Pi^{1}_{d+1}$. Consider consider two operations from ${\rm Config}(\mathbb{R}^{d+1}) \to {\rm Config}(\mathbb{R}^{d})$, denoted $ \pi_{1}$ and $\pi_{\phi}$. The operation $\pi_{1}$ is defined as projecting a point onto $\mathbb{R}^{d}$, if the last component of the position is smaller than $1$. The operation $\pi_{\phi}$ is defined as projecting a point onto $\mathbb{R}^{d}$, if the last component of the position is smaller than $\phi_{S}$ at that point. Note that the push-forward of $\Pi^{1}_{d+1}$ by $\pi_{1}$ is a Poisson Point Process of intensity $1$, and the push-forward of $\Pi^{1}_{d+1}$ by $\pi_{\phi}$ is a Poisson Point Process of intensity $\phi_{S}$. Note also that the pushforward of the pushforward of $\Pi^{1}_{d+1}$ by $\pi_{\phi}$ by 
\begin{equation}
      {\mathbf{F}}_{N, \rm disc }^{S}(C) :=  \left \lfloor{\frac{ S}{T N^{\frac{1}{d}}}}\right \rfloor^{d} \sum_{i} \delta_{ \left(\theta_{N^{\frac{1}{d}}x_{i}} \cdot C \right) \Big|_{\square_S}} 
\end{equation}
has the same distribution as the push-forward of $\mathbf{\Pi}_{N}$ by ${\mathbf{F}}_{N, \rm disc }^{\phi, S}$. 

Let $\delta > 0$. In order for a realization of the pushforward of the pushforward of $\Pi^{1}_{d+1}$ by $\pi_{\phi}$ by $ {\mathbf{F}}_{N, \rm disc }^{S}$ and a realization of the pushforward of the pushforward of $\Pi^{1}_{d+1}$ by $\pi_{1}$ by $ {\mathbf{F}}_{N, \rm disc }^{S}$ to be at distance greater than $\delta$ from each other, there would have to be a fraction of size $\delta$ of squares $N^{\frac{1}{d}}Q^{S}_{i}$ such that $\pi_{1}$ and $\pi_{\phi}$ are different. This implies that there is at least one point in $N^{\frac{1}{d}}Q^{S}_{i} \times (0, \infty)$ such that the last component is between $\min_{x \in N^{\frac{1}{d}}Q^{S}_{i}} \phi_S$ and $\max_{x \in N^{\frac{1}{d}}Q^{S}_{i}} \phi_S$. This event has probability tending to $0$ as $N$ tends to $\infty$ uniformly in $i$, since $\phi$ is uniformly continuous. Denote this probability by $\epsilon(N)$. 
Then the probability that the realizations are at distance greater than $\delta$ from each other is bounded by ${N\choose \delta N}(\epsilon(N))^{\delta N }$, accounting for the possible combinations of squares in which the realizations differ. Recall the elementary bound
\begin{equation}
    {N\choose \delta N} \leq \left( \frac{1}{\delta} \right)^{C N}
\end{equation}
for some absolute constant $C$. Then the probability that the realizations are at distance greater than $\delta$ from each other (denoted $P_{\delta}$) satisfies that
\begin{equation}
    \begin{split}
        \lim_{N \to \infty} \frac{1}{N} \log \left(P_{\delta} \right) &\leq  \lim_{N \to \infty} \frac{1}{N}  \left( C N \log \left( \frac{1}{\delta} \right) + N \log \left(\epsilon(N)\right)  \right)\\
        &= - \infty.
    \end{split}
\end{equation}
Therefore exponential equivalence is proved. 

\textbf{Substep 1.2}[Conclusion of step 1]

By \cite[Proposition 7.5]{leble2017large} and \cite[Theorem 4.2.13]{dembo2009large}, the push-forward of $\mathbf{\Pi}_{N}$ by 
${\mathbf{F}}_{N}^{^{\phi},S}$  satisfies an LDP  at speed $N$ with rate function
\begin{equation}
    \mathcal{F}({\mathbf{P}}) = 
         {\rm Ent}[{\mathbf{P}} | {\mathbf{\Pi}}^{1}].  
\end{equation}

Note that, by Lemma \ref{lem:lebser1}, $d_{\rm Config}({\mathbf{F}}_{N}^{S}(C) ,{\mathbf{F}}_{N}(C) ) $ tends to $0$ uniformly in $C$, as $S$ tends to $\infty$. From this we may conclude the LDP for the push-forward of $\mathbf{\Pi}_{N}$ by 
${\mathbf{F}}_{N}^{^{\phi}}$, and we may conclude step 1. 

Note that, as an immediate consequence, we may infer that, for any $\lambda > 0$,  the push-forward of $\mathbf{\Pi}_{N}$ by 
\begin{equation}
     C \mapsto \frac{1}{ T^{d}} \int_{\mathbb{Q}^{d}} \delta_{\alpha_{\frac{\phi(x)}{\lambda}} \left(\theta_{N^{\frac{1}{d}}x} \cdot C \right)} \, \mathrm d x
\end{equation} satisfies an LDP in $ \mathcal{P}({\rm Config})$ at speed $N$ with rate function
\begin{equation}
    \mathcal{F}({\mathbf{P}}) = 
         {\rm Ent}[{\mathbf{P}} | {\mathbf{\Pi}}^{\lambda}].  
\end{equation}

\textbf{Step 2}[LDP for the tagged empirical field of a Poisson Point Process with variable intensity]

Let $\Lambda_{N} = N^{\frac{1}{d}} \mathbb{Q}^{d}$ and let $\phi$ be a positive continuous function on $\mathbb{Q}^{d}$. Let $\mathbf{\Pi}_{N}$ be a Poisson Point Process on $\Lambda_{N}$ with intensity $\phi (N^{-\frac{1}{d}} x )$. Let $\overline{\mathbf{F}}_{N}$ be defined as 
\begin{equation}
    \overline{\mathbf{F}}_{N}(C) = \frac{1}{ T^{d}} \int_{\mathbb{Q}^{d}} \delta_{ \left( x,\theta_{N^{\frac{1}{d}}x} \cdot C \right)} \, \mathrm d x,
\end{equation}
 Let $\mathfrak{R}_{N}$ be the push-forward of $\mathbf{\Pi}_{N}$ by $\overline{\mathbf{F}}_{N}$. Then we claim that $\mathfrak{R}_{N}$ satisfies an LDP in $ \mathcal{P}_{s}(\mathbb{Q}^{d} \times {\rm Config})$ at speed $N$ with rate function
\begin{equation}
    \mathcal{F}( \overline{\mathbf{P}}) = \begin{cases}
         \overline{\rm Ent}[\overline{\mathbf{P}} | \overline{\mathbf{\Pi}}^{\phi}] \ \ \text{if}\ \  \overline{\mathbf{P}} \in \mathcal{P}_{s}(\Lambda \times {\rm Config}) \\
        \infty \ \ \ \ \ \ \ \ \ \ \ \ \ \, \text{if}\ \  \overline{\mathbf{P}} \notin \mathcal{P}_{s}(\Lambda \times {\rm Config}).
    \end{cases}
\end{equation}
We will prove the claim in a few substeps.

\textbf{Substep 2.1}[Approximation by piece-wise constant functions]

Let $m \in \mathbb{N}$, and let $\{Q_{i}\}_{i \in \{1, ...m\}^{d}}$ be a partition of $\mathbb{Q}^{d}$ into $m^{d}$ hypercubes of size $\frac{T}{m}$. We define the operation $A_{m} : \mathcal{P}_{s}(\mathbb{Q}^{d} \times {\rm Config}) \to \mathcal{P}_{s}(\mathbb{Q}^{d} \times {\rm Config})$ as assigning to a tagged empirical field its approximation by a piecewise constant one on each hypercube: for each $x \in Q_{i}$, $A_{m} ( \overline{\mathbf{P}}) (x, C) $ is defined as
\begin{equation}
    A_{m} ( \overline{\mathbf{P}} ) (x, C) : = \fint_{Q_{i}} \overline{\mathbf{P}} (x, C) \mathrm d x.
\end{equation}
We claim that, for any $\overline{\mathbf{P}} \in \mathcal{P}(\Lambda \times {\rm Config})$, $d_{\mathcal{P}(\Omega \times {\rm Config})} ( \overline{\mathbf{P}}, A_{m} ( \overline{\mathbf{P}} ) ) \leq \frac{T}{m}$. To show this, let $F \in {\rm Lip}_{1}(\Omega \times {\rm Config})$, and define $F_{m}$ as the piece-wise constant approximation of $F$:
\begin{equation}
    F_{m} = \sum_{i \in \{1, ...m\}^{d}} \mathbf{1}_{Q_{i}} \fint_{Q_{i}} F \, \mathrm d x .
\end{equation}
Note that $\|F - F_{m}\|_{L^{\infty}} \leq \frac{T}{m}$. Then 
\begin{equation}
    \begin{split}
         \left| \int F \, \mathrm d ( \overline{\mathbf{P}} - A_{m} (\overline{\mathbf{P}} ) ) \right| &= \left| \int F - F_{m} \, \mathrm d \overline{\mathbf{P}} \right| \\
         &\leq \frac{T}{m},
    \end{split}
\end{equation}
where the last line follows because $\overline{\mathbf{P}}$ is a probability measure. 

\textbf{Substep 2.2}[Conclusion of step 2]

Define $\phi_{m}$ as the piece-wise constant approximation of $\phi$ on the grid $Q_{i}$, and $\overline{\mathbf{F}}_{N}$ as 
\begin{equation}
    \overline{\mathbf{F}}_{N}^{\phi_{m}}(C) = \frac{1}{ T^{d}} \int_{\mathbb{Q}^{d}} \delta_{ \left( x, \alpha_{\frac{\phi(x)}{\phi_{m}(x)}} \theta_{N^{\frac{1}{d}}x} \cdot C \right)} \, \mathrm d x.
\end{equation}
Note that as $m \to \infty$, $d_{\mathcal{P}(\Omega \times {\rm Config})} ( \overline{\mathbf{F}}_{N}(C) ,\overline{\mathbf{F}}_{N}^{\phi_{m}}(C) ) $ tends to $0$ uniformly in $C$ and $N$, since $\frac{\phi(x)}{\phi_{m}(x)}$ is uniformly close to $1$. Denote by 
\begin{equation}
    \epsilon(m) := \sup_{N, C} d_{\mathcal{P}(\Omega \times {\rm Config})} ( \overline{\mathbf{F}}_{N}(C) ,\overline{\mathbf{F}}_{N}^{\phi_{m}}(C) ).  
\end{equation}

Then, for any $\overline{\mathbf{P}} \in \mathcal{P}_{s}(\Lambda \times {\rm Config})$ and $\delta > 0$, 
\begin{equation}
\label{eq:LDPupbound}
    \mathbf{\Pi}_{N} \left( \overline{\mathbf{F}}_{N}(C) \in B \left (\overline{\mathbf{P}}, \delta \right)  \right) \leq \mathbf{\Pi}_{N} \left( A_{m} \left( \overline{\mathbf{F}}_{N}^{\phi_{m}}(C) \right) \in B \left ( A_{m} \left(\overline{\mathbf{P}} \right), \delta + \frac{2T}{M} + \epsilon(m) \right)  \right),
\end{equation}
and similarly, 
\begin{equation}
\label{eq:LDPlowbound2}
    \mathbf{\Pi}_{N} \left( \overline{\mathbf{F}}_{N}(C) \in B \left (\overline{\mathbf{P}}, \delta \right)  \right) \geq \mathbf{\Pi}_{N} \left( A_{m} \left( \overline{\mathbf{F}}_{N}^{\phi_{m}}(C) \right) \in B \left ( A_{m} \left(\overline{\mathbf{P}} \right), \delta - \frac{2T}{M} - \epsilon(m) \right)  \right).
\end{equation}

Letting $N$ tend to $\infty$ in equation \eqref{eq:LDPupbound}, and using step 1 and that $\mathrm{Ent}$ is affine,
\begin{equation}
\label{eq:LDPnearlythere}
    \lim_{N \to \infty} \frac{1}{N} \log \left( \mathbf{\Pi}_{N} \left( \overline{\mathbf{F}}_{N}(C) \in B \left (\overline{\mathbf{P}}, \delta \right)  \right) \right) \leq  \inf_{\overline{\mathbf{P}}' \in B \left ( A_{m} \left(\overline{\mathbf{P}} \right), \delta + \frac{2T}{M} + \epsilon(m) \right) } \overline{\rm Ent}[\overline{\mathbf{P}}' | \overline{\mathbf{\Pi}}^{\phi_{m}}].    
\end{equation}
Letting $m$ tend to $\infty$ in equation \eqref{eq:LDPnearlythere}, 
\begin{equation}
\label{eq:LDPfinally}
    \lim_{N \to \infty} \frac{1}{N} \log \left( \mathbf{\Pi}_{N} \left( \overline{\mathbf{F}}_{N}(C) \in B \left (\overline{\mathbf{P}}, \delta \right)  \right) \right) \leq  \inf_{\overline{\mathbf{P}}' \in B \left ( \overline{\mathbf{P}} , \delta \right) } \overline{\rm Ent}[\overline{\mathbf{P}}' | \overline{\mathbf{\Pi}}^{\phi}].  
\end{equation}
Proceeding similarly with equation \eqref{eq:LDPlowbound2}, 
\begin{equation}
    \lim_{N \to \infty} \frac{1}{N} \log \left( \mathbf{\Pi}_{N} \left( \overline{\mathbf{F}}_{N}(C) \in B \left (\overline{\mathbf{P}}, \delta \right)  \right) \right) \geq  \inf_{\overline{\mathbf{P}}' \in B \left ( \overline{\mathbf{P}} , \delta \right) } \overline{\rm Ent}[\overline{\mathbf{P}}' | \overline{\mathbf{\Pi}}^{\phi}].  
\end{equation}
}

\textbf{Step 3}[Conclusion: from Poisson to Bernoulli]

We now prove the statement of the Theorem. Let $\mu_{\theta}$ be as defined by \eqref{def:theqmeas}. Let $\mathbf{\Pi}_{N}$ be a Poisson Point Process on $N^{\frac{1}{d}} \mathbb{Q}^{d}$ with intensity $\mu_{\theta}(N^{-\frac{1}{d}} x )$. Note that we have, for any $ \overline{\mathbf{Q}} \in \mathcal{P}(\mathbb{Q}^{d} \times {\rm Config})$ and $\delta > 0$,
\begin{equation}
    \mathbf{P}_{N, \beta} \left( \overline{\mathbf{P}}_{N}(X_{N}) \in B( \overline{\mathbf{Q}}, \delta) \right) = \mathbf{\Pi}_{N} \left( \overline{\mathbf{F}}_{N}(C) \in B( \overline{\mathbf{Q}}, \delta) \Big| |C| = N \right).
\end{equation}

Now let $ \overline{\mathbf{P}} \in \mathcal{P}_{s,1}(\mathbb{Q}^{d} \times {\rm Config})$.
{We claim that for any $\epsilon>0$ there exists $\delta$ such that, if $\overline{\mathbf{F}}_{N}(C) \in B( \overline{\mathbf{P}}, \delta)$ then $\left|  \frac{|C|}{N} - 1 \right| < \epsilon$. To prove this claim, note that by Lemma \ref{lem:convden}, if $\overline{\mathbf{F}}_{N}(C) \to \overline{\mathbf{P}}$ then the empirical measure associated to $C$ converges to $ \rho := {\rm int}[ \overline{\mathbf{P}}^{x}]$ weakly in the sense of probability measures. Let ${\rm emp}_{N}(C)$ denote the empirical measure associated to $C$. Then for all $\epsilon > 0$ there exists $\delta$ such that if $\overline{\mathbf{F}}_{N}(C) \in B( \overline{\mathbf{P}}, \delta)$ then $\| {\rm emp}_{N}(C) - \rho \|_{\rm BL} \leq \epsilon$, where $\| \bullet \|_{\rm BL}$ denotes the bounded-Lipschitz metric. In particular, this implies that 
\begin{equation}
    \left|  \frac{|C|}{N} - \int_{\mathbb{Q}^{d}} \rho(x) \, \mathrm d x \right| < \epsilon.  
\end{equation}
Since $\mathcal{P}_{s,1}(\mathbb{Q}^{d} \times {\rm Config})$,
\begin{equation}
    \int_{\mathbb{Q}^{d}} {\rm int}\left[ \overline{\mathbf{P}}^{x}\right] \mathrm d x =1.
\end{equation}
Hence, the claim is proved. For the remainder of the proof, we will treat $\epsilon$ as a function of $\delta$. Note that this function is non-decreasing and satisfies that $\lim_{\delta \to 0}\epsilon = 0$. 
}

We then have that
\begin{equation}
    \begin{split}
        \mathbf{\Pi}_{N} \left( \overline{\mathbf{F}}_{N}(C) \in B( \overline{\mathbf{P}}, \delta) \right) &= \sum_{j=1}^{\infty} \mathbf{\Pi}_{N} \left( \overline{\mathbf{F}}_{N}(C) \in B( \overline{\mathbf{P}}, \delta) \Big| |C| = j \right) \mathbf{\Pi}_{N} \left( |C| = j \right)\\
         &= \sum_{j=(1-\epsilon)N}^{(1+\epsilon)N} \mathbf{\Pi}_{N} \left( \overline{\mathbf{F}}_{N}(C) \in B( \overline{\mathbf{P}}, \delta) \Big| |C| = j \right) \mathbf{\Pi}_{N} \left( |C| = j \right).
    \end{split}
\end{equation}

Therefore
\begin{equation}
\begin{split}
        &\mathbf{\Pi}_{N} \left( \overline{\mathbf{F}}_{N}(C) \in B( \overline{\mathbf{\mathbf{P}}}, \delta) \right) \\
         \leq &\mathbf{\Pi}_{N} \left( \frac{|C|}{N} \in (1-\epsilon, 1+\epsilon) \right) \max_{\frac{j}{N} \in (1-\epsilon, 1+\epsilon)} \mathbf{\Pi}_{N} \left( \overline{\mathbf{F}}_{N}(C) \in B( \overline{\mathbf{P}}, \delta) \Big| |C| = j \right),
\end{split}         
\end{equation}
and similarly, 
\begin{equation}
\begin{split}
        &\mathbf{\Pi}_{N} \left( \overline{\mathbf{F}}_{N}(C) \in B( \overline{\mathbf{P}}, \delta) \right) \\
         \geq & \mathbf{\Pi}_{N} \left( \frac{|C|}{N} \in (1-\epsilon, 1+\epsilon) \right) \min_{\frac{j}{N} \in (1-\epsilon, 1+\epsilon)} \mathbf{\Pi}_{N} \left( \overline{\mathbf{F}}_{N}(C) \in B( \overline{\mathbf{P}}, \delta) \Big| |C| = j \right).
\end{split}         
\end{equation}

Note that
\begin{equation}
    \begin{split}
        \mathbf{\Pi}_{N}(|C|=N) &\leq \mathbf{\Pi}_{N} \left( \frac{|C|}{N} \in (1-\epsilon, 1+\epsilon) \right) \\
        & \leq 1.
    \end{split}
\end{equation}
We recall Stirling's approximation:
\begin{equation}\label{eq:Stirling}
    \log(N!) = N\log(N) - N + O(\log(N)).
\end{equation}
Using equation \eqref{eq:Stirling}, and the definition of a Poisson Point Process (equation \eqref{eq:Poisson}), we have that
\begin{equation}
    \lim_{N \to \infty} \frac{1}{N} \log \left( \mathbf{\Pi}_{N}(|C|=N) \right) =0,
\end{equation}
which implies that
\begin{equation}
    \lim_{\delta \to 0} \lim_{N \to \infty} \frac{1}{N} \log \left( \mathbf{\Pi}_{N} \left( \frac{|C|}{N} \in (1-\epsilon, 1+\epsilon) \right) \right)=0.
\end{equation}

On the other hand, from Definitions \ref{def:distconfig} and \ref{def:disttagged}, we can see that adding or deleting $\epsilon N$ points has a negligible effect on the tagged empirical field as $\epsilon \to 0$, and so
\begin{equation}
    \begin{split}
        &\lim_{\epsilon \to 0} \max_{\frac{j}{N} \in (1-\epsilon, 1+\epsilon)}  \mathbf{\Pi}_{N} \left( \overline{\mathbf{F}}_{N}(C) \in B( \overline{\mathbf{P}}, \delta) \Big| |C| = j \right) \\
        =& \lim_{\epsilon \to 0} \min_{\frac{j}{N} \in (1-\epsilon, 1+\epsilon)} \mathbf{\Pi}_{N} \left( \overline{\mathbf{F}}_{N}(C) \in B( \overline{\mathbf{P}}, \delta) \Big| |C| = j \right)\\
        =&\mathbf{\Pi}_{N} \left( \overline{\mathbf{F}}_{N}(C) \in B( \overline{\mathbf{P}}, \delta) \Big| |C| = N \right).
    \end{split}
\end{equation}

Therefore
\begin{equation}
    \begin{split}
         \lim_{\delta \to 0} \lim_{N \to \infty} \frac{1}{N} \log \left(  \mathbf{P}_{N, \beta} \left( \overline{\mathbf{P}}_{N}(X_{N}) \in B( \overline{\mathbf{P}}, \delta) \right) \right) &=  \lim_{\delta \to 0} \lim_{N \to \infty} \frac{1}{N} \log \left(  \mathbf{\Pi}_{N} \left( \overline{\mathbf{F}}_{N}(C) \in B( \overline{\mathbf{P}}, \delta) \right) \right)\\
         &= - \overline{\rm Ent}[ \overline{\mathbf{P}}| \overline{\mathbf{\Pi}}^{\mu_{\theta}}].
    \end{split}
\end{equation}

\end{proof}

\section{Proof of Theorem \ref{LDPbeta1N}, upper bound}

We now turn to the proof of the upper bound of Theorem~\ref{LDPbeta1N}. The proof is basically a consequence of the compactness and lower semi-continuity results in section \ref{sect:comp}, and Proposition \ref{LDPnoninteracting}. 

\begin{proof}[Proof of Theorem \ref{LDPbeta1N}, upper bound]

We will prove the upper bound of a weak LDP. It is standard (see, for example, \cite{leble2017large}) that this result, along with a lower bound {on $\lim_{\delta \to 0} \lim_{N \to \infty} \\ \left( \frac{1}{N} \log \left( \mathbf{P}_{N, \beta} \left( \overline{\mathbf{P}}_{N} \in B( \overline{\mathbf{P}}, \delta) \right) \right) \right)$ (see Section \ref{sect:lowerbound})} and exponential tightness (which is easy consequence of the fact that the total number of particles is bounded by $N$), implies the full LDP. We thus need to show that for any $ \overline{\mathbf{P}} \in \mathcal{P}_{s,1} (\mathbb{Q}^{d} \times {\rm Config})$ we have 
\begin{equation}
    \lim_{\delta \to 0} \lim_{N \to \infty} \left( \frac{1}{N} \log \left( \mathbf{P}_{N, \beta} \left( \overline{\mathbf{P}}_{N} \in B( \overline{\mathbf{P}}, \delta) \right) \right) \right) \leq -\theta \mathcal{E}(\rho- \mu_{\theta}) - \overline{\rm Ent}[ \overline{\mathbf{P}} | \overline{\mathbf{\Pi}}^{\mu_{\theta}} ],
\end{equation}
where $\rho (x)= {\rm int}[\overline{\mathbf{P}}^{x}]$.

Here comes the argument: using the thermal splitting formula (equation \eqref{eq:thermspltfrm}) and bounding the integral by its maximum value, we have that 
\begin{equation}
\begin{split}
        &\mathbf{P}_{N, \beta} \left( \overline{\mathbf{P}}_{N} \in B( \overline{\mathbf{P}}, \delta) \right) \\
        =& \int_{ \overline{\mathbf{P}}_{N} \in B( \overline{\mathbf{P}}, \delta)} \frac{1}{Z_{N, \beta}} \exp \left( -\beta \mathcal{H}_{N}(X_{N}) \right) \, \mathrm d X_{N}\\
        \leq & \frac{1}{K_{N, \beta}} \inf_{ \overline{\mathbf{P}}_{N} \in B(\overline{\mathbf{P}}, \delta)} \left\{ \exp \left( -N \theta {\rm F}_{N} ({\rm emp}_{N} , \mu_{\theta}) \right) \right\} \int_{ \overline{\mathbf{P}}_{N} \in B(\overline{\mathbf{P}}, \delta)}  \Pi_{i=1}^{N} \mu_{\theta}(x_{i}) \, \mathrm d x_{i}.
\end{split}        
\end{equation}

We recall that, by Lemma \ref{limitofpartitionfunction},
\begin{equation}
    \lim_{N \to \infty} \frac{\log \left( K_{N, \beta} \right)}{N}  = 0.
\end{equation}
Also, by Proposition \ref{LDPnoninteracting}
\begin{equation}
   \lim_{\delta \to 0} \lim_{N \to \infty} \frac{1}{N} \log \left(  \int_{ \overline{\mathbf{P}}_{N}  \in B( \overline{\mathbf{P}}, \delta)} \Pi_{i=1}^{N} \mu_{\theta}(x_{i}) \, \mathrm d x_{i} \right) \leq - \overline{{\rm Ent}}[ \overline{\mathbf{P}} | \overline{\mathbf{\Pi}}^ { \mu_{\theta}}].
\end{equation}

For the remaining term, let $X_{N}^{\delta}$ be defined as 
\begin{equation}
   X_{N}^{\delta} = {\rm argmin}_{ \overline{\mathbf{P}}_{N} \in  B ( \overline{\mathbf{P}}, \delta)} {\rm F}_{N}(X_{N}, \mu_{\theta}).
\end{equation}
We assume for simplicity that the minimum is achieved. Otherwise, we would repeat the argument up to an arbitrarily small error. We also define
\begin{equation}
    {\rm emp}_{N}^{\delta} = {\rm emp}_{N}(X_{N}^{\delta}).
\end{equation}

Then, as $N \to \infty,$ we have that 
\begin{equation}
     {\rm emp}_{N}^{\delta} \to \mu_{\delta}
\end{equation}
weakly in the sense of probability measures for some $\mu_{\delta}$. By Lemma~\ref{lem:l.s.c.} and Remark \ref{rem:thembounded}, we have that
\begin{equation}
    \mathcal{E}(\mu_{\delta} - \mu_{\theta}) \leq \liminf_{N \to \infty} {\rm F}_{N} ({\rm emp}_{N} , \mu_{\theta}).
\end{equation}

Furthermore, as $\delta \to 0,$ we have by Lemma~\ref{lem:convden} that $\mu_{\delta} \to \rho$ weakly in the sense of probability measures. Note that weak positive definiteness implies that $\mathcal{E}(\mu)$ is lower semi-continuous. Therefore
\begin{equation}
    \mathcal{E}(\rho- \mu_{\theta}) \leq \liminf_{\delta \to 0} \mathcal{E}(\mu_{\delta} - \mu_{\theta}).
\end{equation}

Putting everything together, we have that
\begin{equation}
    \lim_{\delta \to 0} \lim_{N \to \infty} \left( \frac{1}{N} \log \left( \mathbf{P}_{N, \beta} \left( \overline{\mathbf{P}}_{N} \in B( \overline{\mathbf{P}}, \delta) \right) \right) \right) \leq -\theta \mathcal{E}(\rho - \mu_{\theta}) - \overline{\rm Ent}[ \overline{\mathbf{P}} | \overline{\mathbf{\Pi}}^{\mu_{\theta}} ].
\end{equation}

\end{proof}

\section{Proof of Theorem \ref{LDPbeta1N}, lower bound}
\label{sect:lowerbound}

Having proved the upper bound of Theorem \ref{LDPbeta1N}, we turn to prove the lower bound. {The proof has the same spirit as the proofs of \cite{leble2017large, leble2017largetwo, leble2017local, hardin2018large, armstrong2021local}. Namely, we will construct a family of configurations with the right energy and enough volume. But unlike the references just mentioned, we need to deal with the mean-field energy of a general interaction and not the renormalized energy of a Riesz interaction.} The crucial ingredient in the proof is the following proposition:
\begin{proposition}
\label{Prop:fund}
Assume that $\beta = \frac{\theta}{N}$ for fixed $\theta > 0$, that $g: \mathbb{R}^{d} \to \mathbb{R}$ satisfies items $1-6$ of Theorem \ref{LDPbeta1N}, and that $V$ satisfies items $1-3$ of Theorem \ref{LDPbeta1N}. Define $ \overline{\mathbf{P}}_{N}$ by Definition~\ref{def:tagempfiel} with $\Omega = \mathbb{Q}^{d}$ and $\mu_{\theta}$ by \eqref{def:theqmeas}.  Let $ \overline{\mathbf{P}} \in \mathcal{P}_{s,1}(\mathbb{Q}^{d} \times {\rm Config})$ be such that ${\rm int}[ \overline{\mathbf{P}}^{x}] \in L^{\infty}$. Then for any ${\rm err}, \delta >0$, there exists a family of configurations $\Lambda_{N} \subset \mathbb{Q}^{d \times N}$ (depending on the previous parameters) such that:
\begin{itemize}
    \item[1.] 
    
    \begin{equation}
    \sup_{X_{N} \in \Lambda_{N}}  d_{\mathcal{P}(\mathbb{Q}^{d} \times {\rm Config})} ( \overline{\mathbf{P}}_{N}(X_{N}), \overline{\mathbf{P}}) \leq \delta.
\end{equation}
    
    \item[2.]
    
    \begin{equation}
        \lim_{N \to \infty} \frac{1}{N} \log \left( \mu_{\theta}^{\otimes N} (\Lambda_{N}) \right) \geq -\overline{\rm Ent}[ \overline{\mathbf{P}} | \overline{\mathbf{\Pi}}^{\mu_{\theta}}] - {\rm err}.
    \end{equation}
\end{itemize}

Furthermore,
  \begin{equation}
         \lim_{\delta \to 0} \lim_{N \to \infty} {\sup_{X_N \in \Lambda_{N}} } \left| {\rm F}_{N}({\rm emp}_{N} , \mu_{\theta}) - \mathcal{E}(\rho - \mu_{\theta}) \right| =0,
    \end{equation}
    where $\rho \in \mathcal{P}(\mathbb{Q}^{d})$ is such that $\rho(x) = {\rm int}[ \overline{\mathbf{P}}^{x}]$.
\end{proposition}

The proof of Proposition \ref{Prop:fund} is found in section \ref{sect:ProofofProp}. We will now prove the lower bound of Theorem \ref{LDPbeta1N} using Proposition \ref{Prop:fund}. 

\begin{proof}[Proof of Theorem \ref{LDPbeta1N}, lower bound]

With the help of Proposition \ref{Prop:fund}, we will prove a weak LDP, namely, we will prove that for any $ \overline{\mathbf{P}} \in \mathcal{P}_{s,1}(\mathbb{Q}^{d} \times {\rm Config})$ {there holds,
\begin{equation}
\label{eq:LDPlowbound}
     \lim_{\delta \to 0} \lim_{N \to \infty} \frac{1}{N} \log \left( \mathbf{P}_{N, \beta} ( \overline{\mathbf{P}}_{N} \in B( \overline{\mathbf{P}}, \delta)) \right) \geq  -\theta \mathcal{E}(\rho - \mu_{\theta}) - \overline{\rm Ent}[ \overline{\mathbf{P}} | \overline{\mathbf{\Pi}}^{\mu_{\theta}} ].
\end{equation}
Note that it is enough to prove equation \eqref{eq:LDPlowbound} for $ \overline{\mathbf{P}} \in \mathcal{P}_{s,1}(\mathbb{Q}^{d} \times {\rm Config})$ such that ${\rm int}[\overline{\mathbf{P}}]^{x} \in L^{\infty}$ (as a function of $x$), since such tagged point processes are dense in $\mathcal{P}_{s,1}(\mathbb{Q}^{d} \times {\rm Config})$ (see \cite{leble2017large}).
}

To this end, let ${\rm err}, \delta >0$ and let $\Lambda_{N}$ be as in Proposition \ref{Prop:fund}, with
\begin{equation}
    \sup_{X_{N} \in \Lambda_{N}}  d_{\mathcal{P}(\mathbb{Q}^{d} \times {\rm Config})} ( \overline{\mathbf{P}}_{N}(X_{N}), \overline{\mathbf{P}}) \leq \delta
\end{equation}
and 
    \begin{equation}
     \lim_{N \to \infty} \frac{1}{N} \log \left( \mu_{\theta}^{\otimes N} (\Lambda_{N}) \right) \geq -\overline{\rm Ent}[ \overline{\mathbf{P}} | \overline{\mathbf{\Pi}}^{\mu_{\theta}}] - {\rm err}.
    \end{equation}
Then
\begin{equation}
    \begin{split}
        \mathbf{P}_{N, \beta} ( \overline{\mathbf{P}}_{N} \in B( \overline{\mathbf{P}}, \delta)) &\geq \frac{1}{K_{N, \beta}} \int_{\Lambda_{N}} \exp \left( - { N \theta}  \left[ {\rm F}_{N}({\rm emp}_{N} , \mu_{\theta}) \right] \right) \, \mathrm d \mu_{\theta}^{\otimes N}\\
        &\geq \frac{1}{K_{N, \beta}} \mu_{\theta}^{\otimes N} (\Lambda_{N})  \exp \left( -{ N \theta}  \sup_{X_{N} \in \Lambda_{N}}\left[ {\rm F}_{N}({\rm emp}_{N} , \mu_{\theta})  \right] \right).
    \end{split}
\end{equation}

Therefore for any ${\rm err}, \delta >0$,  we have that
\begin{equation}
     \lim_{N \to \infty} \frac{1}{N} \log \left( \mathbf{P}_{N, \beta} ( \overline{\mathbf{P}}_{N} \in B( \overline{\mathbf{P}}, \delta)) \right) \geq -\theta \mathcal{E}(\rho - \mu_{\theta}) - \overline{\rm Ent}[ \overline{\mathbf{P}} | \overline{\mathbf{\Pi}}^{\mu_{\theta}}] - {\rm err}.
\end{equation}

Since ${\rm err} $ is arbitrary, we can conclude that
\begin{equation}
     \lim_{N \to \infty} \frac{1}{N} \log \left( \mathbf{P}_{N, \beta} ( \overline{\mathbf{P}}_{N} \in B( \overline{\mathbf{P}}, \delta)) \right) \geq -\theta \mathcal{E}(\rho - \mu_{\theta}) - \overline{\rm Ent}[ \overline{\mathbf{P}} | \overline{\mathbf{\Pi}}^{\mu_{\theta}}],
\end{equation}
which implies the desired result. 

\end{proof}

\section{Proof of Proposition \ref{Prop:fund}}
\label{sect:ProofofProp}

This section is devoted to proving Proposition \ref{Prop:fund}.

\begin{proof}
The construction is basically the same as the one found in 
\cite{leble2017large, hardin2018large}, but the energy estimate is essentially different since we are dealing with the mean-field energy of a general interaction, not the renormalized energy of a Riesz interaction.

\textbf{Step 1}[Generating microstates]

Consider $N^{\frac{1}{d}} \mathbb{Q}^{d}$, and divide it into smaller hypercubes $\{K_{i}\}_{i \in I}$ of sidelength $R$, for some $R>0$ to be determined later. The following statement is a close adaptation of Step 1 in the proof of Proposition 4.4 of \cite{hardin2018large}:

We claim that for any $\delta > 0$, $S, R >0$, and $N >1$ there exists a family $\mathcal{A} = \mathcal{A}(\delta, S, R, N)$ of point configurations $C$ such that: 
\begin{itemize}
    \item[1.] 
    \begin{equation}
        C = \sum_{i \in I} C_{i},
    \end{equation}
    where $C_{i}$ is a point configuration on $K_{i}$.
    
    \item[2.] $|C|= N$.
    
    \item[3.] The associated tagged empirical field is close to $ \overline{\mathbf{P}}$:
    \begin{equation}
        \overline{\mathbf{P}}_{N}(C) \in B( \overline{\mathbf{P}}, \delta),
    \end{equation}
    where
    \begin{equation}
         \overline{\mathbf{P}}_{N}(C) = \frac{1}{T^{d}} \int_{\mathbb{Q}^{d}} \delta_{ \left( x, \theta_{N^{\frac{1}{d}}x}C \right)} \, \mathrm d x.
    \end{equation}
    
    \item[4.] The volume of $\mathcal{A}$ satisfies, for any $\delta>0$,
    \begin{equation}
        \liminf_{S \to \infty} \liminf_{R \to \infty} \lim_{N \to \infty} \frac{1}{N} \log \left( \mu_{\theta}^{\otimes N} (\alpha_{N^{-\frac{1}{d}}} \mathcal{A}) \right) \geq -\overline{\rm Ent}[ \overline{\mathbf{P}} | \overline{\mathbf{\Pi}}^{\mu_{\theta}} ], 
    \end{equation}
    where (as defined before), $\alpha_{\lambda}: {\rm Config} \to {\rm Config}$ is a dilation by $\lambda$.
    
    \item[5.] For each $i \in I$ we have
    \begin{equation}
        |C_{i}| \leq 2 \|\mu_{\theta}\|_{L^{\infty}} S R^{d}.
    \end{equation}
\end{itemize}

We have omitted item $4$ in the analogous statement in \cite{hardin2018large} since it is not relevant for our purpose. Other than that, the only difference with respect to \cite{hardin2018large} is that in items $4$ and $5$ (or rather, their analogue in \cite{hardin2018large}) the reference measure is the uniform measure; i.e. in our notation, items $4$ and $5$ in \cite{hardin2018large} are stated as follows:
\begin{itemize}
        \item[4.] The volume of $\mathcal{A}$ satisfies, for any $\delta>0$,
    \begin{equation}
        \liminf_{S \to \infty} \liminf_{R \to \infty} \lim_{N \to \infty} \frac{1}{N} \log \left( \mu^{\otimes N} (\alpha_{N^{-\frac{1}{d}}} \mathcal{A}) \right) \geq -\overline{\rm Ent}[ \overline{\mathbf{P}} | \overline{\mathbf{\Pi}}^{\mu} ], 
    \end{equation}
    where $\mu$ denotes the uniform probability measure on $\mathbb{Q}^{d}$. 
    
    \item[5.] For each $i \in I$ we have
    \begin{equation}
        |C_{i}| \leq \frac{2S R^{d}}{T}.
    \end{equation}
\end{itemize}

The proof of our statement is basically the same as the proof of the analogous statement in \cite{hardin2018large}, which is basically the same as the proof of Lemma 6.3 in \cite{leble2017large}, which is by construction. Hence, we will not give a full proof, but rather a sketch, and indicate the differences with respect to the proof of Lemma 6.3 in \cite{leble2017large}. The idea of the construction in the proof of of Lemma 6.3 in \cite{leble2017large} is to draw points from a Poisson Point Process of uniform intensity, and keep the ones that satisfy item $3$. Proposition 4.1 in \cite{leble2017large} (which is a particular case of Proposition \ref{LDPnoninteracting} for a Poisson Point Process of constant intensity) then implies that the volume of this family of point configurations satisfies item $4$. The authors then modify the point configurations so that they satisfy items $1,2$, and $5$, without changing the volume of this family too much. The proof in our case is basically the same, the only difference in our case is that, for each $i$, we draw the point configuration in $K_{i}$ from an in-homogeneous Poisson Point Process of intensity $\mu_{\theta}(\cdot N^{-\frac{1}{d}})|_{K_{i}}$, and Proposition \ref{LDPnoninteracting} then implies that the volume of this family satisfies item $4$. The rest of the proof is the same. 

\textbf{Step 2}[Regularization]

We then apply the regularization procedure described in \cite{leble2017large} Lemma 5.11 and \cite{hardin2018large} Proposition 4.4. As in step 1, we will not give a full proof, but rather indicate the main steps, since this is essentially the same construction as in \cite{leble2017large}. The goal of the regularization procedure is to modify the configurations so that no two points are too close together, while not changing the volume of the point configurations or the associated tagged empirical field too much. The regularization procedure is defined as follows: 
\begin{itemize}
    \item[1.] We partition $N^{\frac{1}{d}} \mathbb{Q}^{d}$ into smaller hypercubes of side length $6 \tau$, for $\tau$ to be determined later.  
    
    \item[2.] If one of these hypercubes $\mathcal{K}$ contains more than one point or if it contains a point and one of the adjacent hypercubes also contains a point, we replace the point configuration in $\mathcal{K}$  by one with the same number of points but confined in the central, smaller hypercube $\mathcal{K}' \subset \mathcal{K}$  of side length $3 \tau$ and that lives on a lattice (the spacing of the lattice depends on the initial number of points in $\mathcal{K}$).
\end{itemize}

Figure \ref{fig:regular} shows the effect of the regularization procedure on a point configuration. 
\begin{figure}
    \centering
    \includegraphics[scale=0.4]{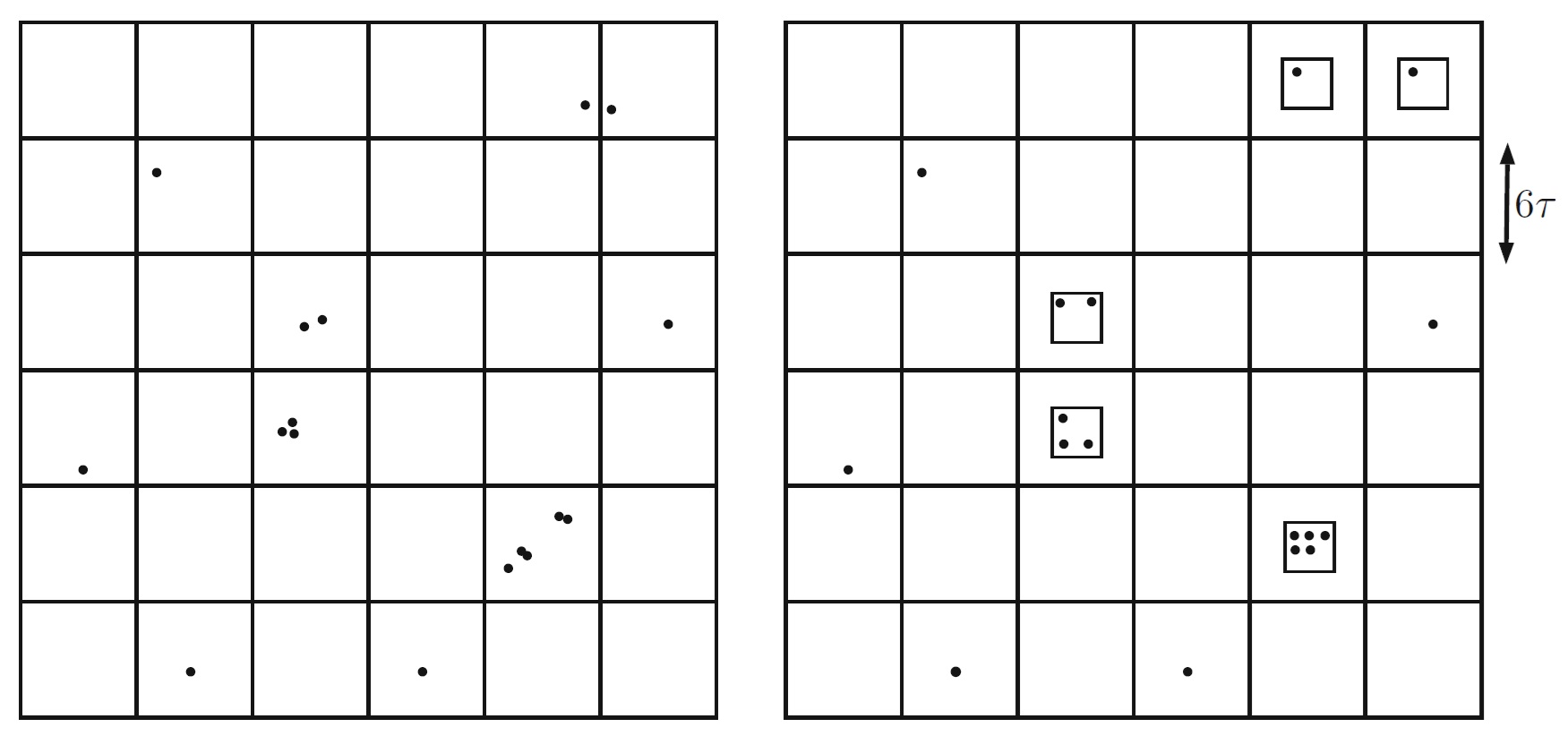}
    \caption{An unregularized configuration (left) and the regularization procedure applied to it (right). Taken from \cite{leble2017large}.}
    \label{fig:regular}
\end{figure}

The set $\Lambda_{N}$ is now defined as $\Lambda_{N} = \alpha_{N^{-\frac{1}{d}}} \mathcal{R}\mathcal{A}$, where the family of point configurations $\mathcal{R}\mathcal{A}$ consists of the regularization procedure applied to each point configuration in $\mathcal{A}$.

By Lemma 5.9 of \cite{leble2017large}, we have that for any $i \in I$,
\begin{equation}
\label{eq:tau}
    \lim_{\tau \to 0} \sup_{C \in {\rm Config}(K_{i})} d_{\rm Config} (C, \mathcal{R}C) =0.
\end{equation}

On the other hand, by Claim 6.8 of \cite{leble2017large} we have that
\begin{equation}
\label{eq:tauR}
    \limsup_{\tau \to 0, R \to \infty, N \to \infty} \sup_{C \in \mathcal{A}}  d_{\mathcal{P}(\Omega \times {\rm Config})} ( \overline{\mathbf{P}}_{N}(C), \overline{\mathbf{P}}_{N}(\mathcal{R}C)) - \frac{1}{|I|} \sum_{i \in I}  d_{\rm Config} (\theta_{x_{i}} \cdot C, \theta_{x_{i}} \cdot \mathcal{R}C ) =0,
\end{equation}
where $x_{i}$ is the center of the hypercube $K_{i}$. Putting together equations \eqref{eq:tau} and \eqref{eq:tauR}, we have that
\begin{equation}
    \limsup_{\tau \to 0, R \to \infty, S \to \infty} \sup_{C \in \Lambda_{N}}  d_{\mathcal{P}(\Omega \times {\rm Config})} ( \overline{\mathbf{P}}_{N}(C), \overline{\mathbf{P}}) \leq \delta.
\end{equation}

Using Lemmas 6.10-6.16 in \cite{leble2017large}, we have that 
\begin{equation}
    \begin{split}
         \lim_{\tau \to 0} \liminf_{S \to \infty} \liminf_{R \to \infty} \lim_{N \to \infty} \frac{1}{N} \log \left( \mu_{\theta}^{\otimes N} (\Lambda_{N}) \right) &\geq \liminf_{S \to \infty} \liminf_{R \to \infty} \lim_{N \to \infty} \frac{1}{N} \log \left( \mu_{\theta}^{\otimes N} (\alpha_{N^{-\frac{1}{d}}} \mathcal{A}) \right)
        \\
        &\geq -\overline{\rm Ent}[ \overline{\mathbf{P}} | \overline{\mathbf{\Pi}}^{\mu_{\theta}} ]. 
    \end{split}
\end{equation}

Hence, for any ${\rm err}>0$, we can find $R, S, \tau$ such that
    \begin{equation}
         \lim_{N \to \infty} \frac{1}{N} \log \left( \mu_{\theta}^{\otimes N} (\Lambda_{N}) \right) \geq -\overline{\rm Ent}[ \overline{\mathbf{P}} | \overline{\mathbf{\Pi}}^{\mu_{\theta}}] - {\rm err}.
    \end{equation}
    
Items $1$ and $2$ are now proved, we move to item $3$, which requires us to estimate the energy of such configurations. 

\textbf{Step 3}[Energy estimate]

Throughout the proof, we will use the notation $\delta_{x}^{\epsilon}$ for the uniform probability measure on $ B (x,\epsilon)$, and we will also use the notation $\delta_{\epsilon}= \delta_{0}^{\epsilon}$.

\textbf{Substep 3.1}

Let ${\rm emp}_{N}^{\epsilon} =  {\rm emp}_{N} \ast \delta_{\epsilon}$, for $\epsilon>0$ to be determined later. We will first derive an estimate for
\begin{equation}
    \left| {\rm F}_{N} \left( {\rm emp}_{N} , \mu_{\theta} \right) - \mathcal{E} \left( {\rm emp}_{N}^{\epsilon} - \mu_{\theta} \right) \right|.
\end{equation}

For this, we write
\begin{equation}
\begin{split}
    &\left| {\rm F}_{N} \left( {\rm emp}_{N} , \mu_{\theta} \right) - \mathcal{E} \left( {\rm emp}_{N}^{\epsilon} - \mu_{\theta} \right) \right| \\
    =&\left| {\rm F}_{N} \left( {\rm emp}_{N} \right) , \mathcal{E} \left( {\rm emp}_{N}^{\epsilon} \right) + 2 \mathcal{G} \left( {\rm emp}_{N} - {\rm emp}_{N}^{\epsilon}, \mu_{\theta} \right) \right|  \\
    \leq& 2 \left|  \mathcal{G}\left( {\rm emp}_{N} - {\rm emp}_{N}^{\epsilon}, \mu_{\theta} \right) \right| + \left|  \mathcal{G}^{\neq} \left( {\rm emp}_{N} - {\rm emp}_{N}^{\epsilon},  {\rm emp}_{N} \right) \right|+ \left|  \mathcal{G} \left( {\rm emp}_{N} - {\rm emp}_{N}^{\epsilon},  {\rm emp}_{N}^{\epsilon} \right) \right|\\
     &\ \ \ \ \  + 2 \left|  \mathcal{G}\left( {\rm emp}_{N} - {\rm emp}_{N}^{\epsilon}, \mu_{\theta} \right) \right| + \left|  \sum_{i \neq j} \int h^{\delta_{x_{i}} - \delta_{x_{i}}^{\epsilon}} \delta_{x_{j}} \right|+ \left|  \sum_{i \neq j} \int h^{\delta_{x_{i}} - \delta_{x_{i}}^{\epsilon}} \delta_{x_{j}}^{\epsilon} \right| + \frac{1}{N} \mathcal{E}(\delta^{\epsilon})
\end{split}
\end{equation}

First we deal with the term $\left|  \mathcal{G} \left( {\rm emp}_{N} - {\rm emp}_{N}^{\epsilon}, \mu_{\theta} \right) \right|$. Note that in order to show that \\ $\left|  \mathcal{G} \left( {\rm emp}_{N} - {\rm emp}_{N}^{\epsilon}, \mu_{\theta} \right) \right| \to 0$, it suffices to show that $\left|  \mathcal{G} \left( \delta_{x} - \delta_{x}^{\epsilon}, \mu_{\theta} \right) \right| \to 0$ uniformly in $x$. To this end, take any $r>0$ and write 
\begin{equation}
    \mathcal{G} \left( \delta_{x} - \delta_{x}^{\epsilon}, \mu_{\theta} \right) = \int_{|y-x| \leq r} h^{\delta_{x}- \delta_{x}^{\epsilon}} \mu_{\theta}(y) \, \mathrm d y + \int_{|y-x| \geq r} h^{\delta_{x}- \delta_{x}^{\epsilon}} \mu_{\theta}(y) \, \mathrm d y. 
\end{equation}

Note that
\begin{equation}\label{eq:firstterm}
   \left| \int_{|y-x| \leq r} h^{\delta_{x}- \delta_{x}^{\epsilon}} \mu_{\theta}(y) \, \mathrm d y \right| \leq 2 \| \mu_{\theta} \|_{L^{\infty}} \int_{|y-x| \leq r + \epsilon} g(y) \, \mathrm d y.
\end{equation}

For the remainder of the proof, we will use the notation {
\begin{equation}
    I(\epsilon) = \int_{\square_{\epsilon}} D(x) \, \mathrm d x,
\end{equation}
where $D$ is an in hypothesis $6$ of $g$ in Theorem \ref{LDPbeta1N}, }and also the function
\begin{equation}
    \Psi(\alpha,\beta) = \sup_{|x-y|<\beta, |x|>\alpha, |y|>\alpha} \left| g(x) - g(y) \right|.
\end{equation}

Note that $\Psi(\alpha,\beta)$ is decreasing in $\alpha$, increasing in $\beta$, and for every $\alpha$, $\Psi(\alpha,\beta) \to 0$ as $\beta \to 0$ by uniform continuity (Property $3$ of $g$). 
We then have that 
\begin{equation}\label{eq:secondterm}
   \left|  \int_{|y-x| \geq r} h^{\delta_{x}- \delta_{x}^{\epsilon}} \mu_{\theta}(y) \, \mathrm d y \right| \leq \| \mu_{\theta} \|_{L^{1}} \Psi(r-\epsilon, \epsilon),
\end{equation}
since for any $y$ such that $|x-y|>r$, we have that
\begin{equation}
    \left| h^{\delta_{x}- \delta_{x}^{\epsilon}} (y) \right| \leq \Psi(r-\epsilon, \epsilon).
\end{equation}

Putting together equations \eqref{eq:firstterm} and \eqref{eq:secondterm}, we have that
\begin{equation}
    \left| \mathcal{G} \left( \delta_{x} - \delta_{x}^{\epsilon}, \mu_{\theta} \right) \right| \leq   \| \mu_{\theta} \|_{L^{1}} \Psi(r-\epsilon, \epsilon) + \| \mu_{\theta} \|_{L^{\infty}} I(r+\epsilon).
\end{equation}

We will now deal with the term $\mathcal{G}^{\neq} \left( {\rm emp}_{N} - {\rm emp}_{N}^{\epsilon},  {\rm emp}_{N} \right)$. The procedure is similar. Note that it suffices to show that $\mathcal{G}^{\neq} \left( {\rm emp}_{N} - {\rm emp}_{N}^{\epsilon},x \right) \to 0$ uniformly for any $x$. For arbitrary $r>0$, we write
\begin{equation}
    \mathcal{G}^{\neq} \left( \delta_{x} - \delta_{x}^{\epsilon}, {\rm emp}_{N} \right) = \int_{y:0<|y-x| \leq r} h^{\delta_{x}- \delta_{x}^{\epsilon}} \, \mathrm d {\rm emp}_{N}(y)  + \int_{y:|y-x| \geq r} h^{\delta_{x}- \delta_{x}^{\epsilon}} \, \mathrm d {\rm emp}_{N}(y),
\end{equation}
and we have 
\begin{equation}
    \int_{y:|y-x| \geq r} h^{\delta_{x}- \delta_{x}^{\epsilon}} \, \mathrm d {\rm emp}_{N}(y) \leq \Psi(r-\epsilon, \epsilon).
\end{equation}

On the other hand, it is not hard to prove that our construction (as a consequence of the regularization procedure) satisfies that {for any $i \neq j$ the particle positions satisfy that $|x_{i} - x_{j}| \geq c_{S, R, \tau} N^{- \frac{1}{d}}$. Basically, this is a consequence of the fact that the number of particles in each hypercube of side $R$ is bounded above, and so the minimum spacing of the lattice to which the particles are assigned is bounded below. As a consequence, for any $t > 0$ the number of particles at distance greater than $t$ and smaller than $t+N^{ - \frac{1}{d}}$ is bounded above by $C_{S, R, \tau} N^{\frac{d-1}{d}} t^{d-1}$.
Let $\{r_{i}\}$ be a partition of the interval $(0,r)$ of size $N^{- \frac{1}{d}}$. Then
\begin{equation}
    \begin{split}
        \left| \int_{y:0<|y-x| \leq r} h^{\delta_{x}} \, \mathrm d {\rm emp}_{N}(y) \right| &=  \left| \sum_{i} \int_{y:r_{i}<|y-x| \leq r_{i+1}} h^{\delta_{x}} \, \mathrm d {\rm emp}_{N}(y) \right| \\
        &\leq C_{S, R, \tau} \frac{1}{N^{\frac{1}{d}}}\sum_{i} D(r_{i}) r_{i+1}^{d-1}\\
        &\leq C_{S, R, \tau} \int_{0}^{r} D(s) s^{d-1} \, \mathrm d s \\
        & \leq C_{S, R, \tau} I(r).
    \end{split}
\end{equation}

Proceeding analogously, we may prove that 
\begin{equation}
   \left| \int_{y:0<|y-x| \leq r} h^{ \delta_{x}^{\epsilon}} \, \mathrm d {\rm emp}_{N}(y) \right|  \leq C_{S, R, \tau} I(r+\epsilon),
\end{equation}
and so 
\begin{equation}
    \left| \int_{y:0<|y-x| \leq r} h^{\delta_{x}- \delta_{x}^{\epsilon}} \, \mathrm d {\rm emp}_{N}(y) \right| \leq C_{S, R, \tau} I(r+\epsilon).
\end{equation}
} 

Putting everything together, we have that
\begin{equation}
    \mathcal{G}^{\neq} \left( \delta_{x} - \delta_{x}^{\epsilon}, {\rm emp}_{N} \right) \leq  C_{S, R, \tau} I(r+\epsilon) + \Psi(r-\epsilon, \epsilon).
\end{equation}

Similarly, we have that 
\begin{equation}
   \left|  \sum_{i \neq j} \int h^{\delta_{x_{i}} - \delta_{x_{i}}^{\epsilon}} \delta_{x_{j}}^{\epsilon} \right| \leq C_{S, R, \tau} I(r+\epsilon) + \Psi(r-\epsilon, \epsilon),
\end{equation}
where $C_{S, R, \tau}$ depends on $S, R, \tau$.

Adding the errors we have that for any $r>0$, 
\begin{equation}
   \left| {\rm F}_{N} \left( {\rm emp}_{N} , \mu_{\theta} \right) - \mathcal{E} \left( {\rm emp}_{N}^{\epsilon} - \mu_{\theta} \right) \right| \leq C_{S, R, \tau} I(r+\epsilon)+C\Psi(r-\epsilon, \epsilon)+ \frac{1}{N} \mathcal{E}(\delta^{\epsilon}),
\end{equation}
where $C$ is a constant that depends on $\mu_{\theta}$ and $C_{S, R, \tau}$ is a constant that depends, additionally. on $S, R,$ and  $\tau$.

\textbf{Substep 3.2}

We now get an estimate for 
\begin{equation}
    \mathcal{E} \left( {\rm emp}_{N}^{\epsilon} - \rho^{\epsilon} \right), 
\end{equation}
where 
\begin{equation}
    \rho^{\epsilon} = \rho \ast \delta^{\epsilon}.
\end{equation}

For this, let $\eta>0$ to be determined later, and let $K_{i}$ be hypercubes of length $\eta$ which cover $\mathbb{Q}^{d}$, and which are pairwise disjoint except for a set of measure $0$. Let $x \in \mathbb{Q}^{d}$, with $x \in K_{j}$. Note that we can write 
\begin{equation}
\label{eq:splith}
    h^{{\rm emp}_{N}^{\epsilon} - \rho^{\epsilon}} = \int_{K_{j}} g^{\epsilon}(x-y) \, \mathrm d ( {\rm emp}_{N} - \rho) (y) + \sum_{i \neq j} \int_{K_{i}} g^{\epsilon}(x-y) \, \mathrm d ( {\rm emp}_{N} - \rho) (y),
\end{equation}
where
\begin{equation}
    g^{\epsilon} = g \ast \delta^{\epsilon}.
\end{equation}

We will now get an estimate for each of the terms in the RHS of equation \eqref{eq:splith}. Assume W.L.O.G. that
\begin{equation}
     {\rm emp}_{N}(K_{i}) \leq \rho(K_{i}).
\end{equation}
Then we have that 
\begin{equation}
\label{eq:splittingerr}
    \begin{split}
       &\sum_{i \neq j} \left| \int_{K_{i}} g^{\epsilon}(x-y) \, \mathrm d ( {\rm emp}_{N} - \rho) (y) \right|  \\
       \leq &\sum_{i \neq j} \left|  {\rm emp}_{N}(K_{i}) \min_{y \in K_{i}}g^{\epsilon}(x-y) -  \rho(K_{i}) \max_{y \in K_{i}}g^{\epsilon}(x-y)  \right|   \\ 
       \leq &\sum_{i \neq j} \left| \left[ {\rm emp}_{N}(K_{i}) -  \rho(K_{i}) \right] \max_{y \in K_{i}}g^{\epsilon}(x-y)  \right| + \sum_{i \neq j}\left|   \rho(K_{i}) \left[ \max_{y \in K_{i}}g^{\epsilon}(x-y) - \min_{y \in K_{i}}g^{\epsilon}(x-y) \right] \right|.
    \end{split}
\end{equation}

We now introduce the function ${\rm Disc}(\eta, \delta, N)$, defined as 
\begin{equation}
    {\rm Disc}(\eta, \delta, N) = \max_{X_{N} \in \Lambda_{N}} \max_{i}  \left|  {\rm emp}_{N}(K_{i}) -  \rho(K_{i}) \right|.
\end{equation}
Note that for any \emph{fixed} $ \eta >0$, we have that
\begin{equation}
\label{eq:weakconvergence}
   \lim_{\delta \to 0} \lim_{N \to \infty} {\rm Disc}(\eta, \delta, N)  =0,
\end{equation}
since ${\rm emp}_{N}$ converges to $\rho$ weakly in the sense of probability measures. This would not be true if $\eta$ was tending to $0$, since in that case the number of $i$'s would grow with $N$. On the other hand,
\begin{equation}
\begin{split}
    \sum_{i \neq j} \left| \left[ {\rm emp}_{N}(K_{i}) -  \rho(K_{i}) \right] \max_{y \in K_{i}}g^{\epsilon}(x-y)  \right| &\leq  {\rm Disc}(\eta, \delta, N) \sum_{i \neq j} \left| \max_{y \in K_{i}}g^{\epsilon}(x-y)  \right| \\
    &\leq \frac{c_{g, \epsilon}}{\eta^{d}} {\rm Disc}(\eta, \delta, N) ,
\end{split}    
\end{equation}
where $c_{g, \epsilon}$ is defined as 
\begin{equation}
    c_{g, \epsilon} =  \max_{y \in \mathbb{Q}^{d}} \left| g^{\epsilon}(y) \right| < \infty.
\end{equation}

We now turn to the second term in the last line of \eqref{eq:splittingerr}. For this, let $r_{1}>0$ be big enough compared to $\eta$. Then we have that 
\begin{equation}
    \begin{split}
        &\sum_{i \neq j}\left|   \rho(K_{i}) \left[ \max_{y \in K_{i}}g^{\epsilon}(x-y) - \min_{y \in K_{i}}g^{\epsilon}(x-y) \right] \right|  \\
        \leq&\sum_{i \neq j}\left|   \|\rho\|_{L^{\infty}} \eta^{d} \left[ \max_{y \in K_{i}}g^{\epsilon}(x-y) - \min_{y \in K_{i}}g^{\epsilon}(x-y) \right] \right|\\
       =& \sum_{i \neq j, K_{i} \subset B(x,r_{1})}\left|   \|\rho\|_{L^{\infty}} \eta^{d} \left[ \max_{y \in K_{i}}g^{\epsilon}(x-y) - \min_{y \in K_{i}}g^{\epsilon}(x-y) \right] \right|  \\
       &\ \ \ \ \ \ \ \ +\sum_{i \neq j, K_{i} \cap (\mathbb{Q}^{d} \setminus B(x,r_{1})) \neq \phi}\left|   \|\rho\|_{L^{\infty}} \eta^{d} \left[ \max_{y \in K_{i}}g^{\epsilon}(x-y) - \min_{y \in K_{i}}g^{\epsilon}(x-y) \right] \right|.
    \end{split}
\end{equation}

Proceeding as in step 3.1, we have that 
\begin{equation}
\begin{split}
    \sum_{i \neq j, K_{i} \subset B(x,r_{1})}\left|   \|\rho\|_{L^{\infty}} \eta^{d} \left[ \max_{y \in K_{i}}g^{\epsilon}(x-y) - \min_{y \in K_{i}}g^{\epsilon}(x-y) \right] \right| \leq 2  \|\rho\|_{L^{\infty}} I(r_{1}+\epsilon), \\
   \sum_{i \neq j, K_{i} \cap (\mathbb{Q}^{d} \setminus B(x,r_{1})) \neq \phi}\left|   \|\rho\|_{L^{\infty}} \eta^{d} \left[ \max_{y \in K_{i}}g^{\epsilon}(x-y) - \min_{y \in K_{i}}g^{\epsilon}(x-y) \right] \right| \leq  \|\rho\|_{L^{\infty}} \Psi(r_{1}-\eta, \eta).
\end{split}    
\end{equation}

We now estimate the last term in equation \eqref{eq:splith}, namely
\begin{equation}
    \int_{K_{j}} g(x-y) \, \mathrm d ( {\rm emp}_{N}^{\epsilon} - \rho) (y). 
\end{equation}
Proceeding again as in step 3.1, we have that 
\begin{equation}
   \left|  \int_{K_{j}} g(x-y) \, \mathrm d ( {\rm emp}_{N}^{\epsilon} - \rho) (y) \right| \leq (1+C_{S, R, \tau}) I(\eta). 
\end{equation}

Putting everything together,  we have that
\begin{equation}
    h^{ {\rm emp}_{N}^{\epsilon} - \rho^{\epsilon} } \leq \frac{c_{g, \epsilon}}{\eta^{d}} {\rm Disc}(\eta, \delta, N)+ \|\rho\|_{L^{\infty}} \Psi(r_{1}-\eta, \eta) + \|\rho\|_{L^{\infty}} I(r_{1}+\epsilon)+(1+C_{S, R, \tau}) I(\eta).
\end{equation}

\textbf{Substep 3.3}[Conclusion of the energy estimate]

By adding and subtracting terms, we have that for any $\epsilon>0$,
\begin{equation}
    \begin{split}
        & \left| {\rm F}_{N}({\rm emp}_{N} , \mu_{\theta}) - \mathcal{E}(\rho - \mu_{\theta}) \right|  \\
        \leq & \left| {\rm F}_{N}({\rm emp}_{N} , \mu_{\theta}) - \mathcal{E}({\rm emp}_{N}^{\epsilon} - \mu_{\theta}) \right| + \left| \mathcal{E}({\rm emp}_{N}^{\epsilon} - \mu_{\theta}) - \mathcal{E}(\rho - \mu_{\theta}) \right|  \\
        \leq & \left| {\rm F}_{N}({\rm emp}_{N} , \mu_{\theta}) - \mathcal{E}({\rm emp}_{N}^{\epsilon} - \mu_{\theta}) \right| + \left| \mathcal{E}({\rm emp}_{N}^{\epsilon} - \mu_{\theta}) - \mathcal{E}(\rho^{\epsilon} - \mu_{\theta}) \right| \\
        &\ \ \ \ \ \ \ + \left| \mathcal{E}(\rho - \mu_{\theta}) - \mathcal{E}(\rho^{\epsilon} - \mu_{\theta})  \right|.
    \end{split}
\end{equation}

Using polar factorization for the quadratic form $\mu \to \mathcal{E}(\mu)$, we have
\begin{equation}
    \begin{split}
       \left| \mathcal{E}({\rm emp}_{N}^{\epsilon} - \mu_{\theta}) - \mathcal{E}(\rho^{\epsilon} - \mu_{\theta}) \right| &= \left| \mathcal{G} \left( {\rm emp}_{N}^{\epsilon} + \rho^{\epsilon} - 2 \mu_{\theta}, {\rm emp}_{N}^{\epsilon} - \rho^{\epsilon}  \right) \right| \\
        &\leq 4 \|  h^{ {\rm emp}_{N}^{\epsilon} - \rho^{\epsilon} } \|_{L^{\infty}}.
    \end{split}
\end{equation}

Putting together all the previous estimates, we have that, for any $\epsilon, \eta, r, r_{1} >0$,
\begin{equation}
\label{eq:nearlythere}
\begin{split}
    &\lim_{\delta \to 0}  \lim_{N \to \infty}{\sup_{X_N \in \Lambda_{N}} } \left| {\rm F}_{N}({\rm emp}_{N} , \mu_{\theta}) - \mathcal{E}(\rho - \mu_{\theta}) \right|  \\
    \leq &C_{S, R, \tau} I(r+\epsilon)+C\Psi(r-\epsilon, \epsilon) +  \|\rho\|_{L^{\infty}} \Psi(r_{1}-\eta, \eta) + \|\rho\|_{L^{\infty}} I(r_{1}+\epsilon)+\\
    & \ \ \ \ \ \ (1+C_{S, R, \tau}) I(\eta) + \left| \mathcal{E}(\rho - \mu_{\theta}) - \mathcal{E}(\rho^{\epsilon} - \mu_{\theta})  \right| .
\end{split}    
\end{equation}

Taking the limit $\epsilon \to 0$, then $\eta \to 0$ \footnote{Note that there is no issue with taking $\eta \to 0$, since the term ${\rm Disc}(\eta, \delta, N)$ present in equation \eqref{eq:weakconvergence} is no longer present in equation \eqref{eq:nearlythere}}, then $r, r_{1} \to 0$, we have that
\begin{equation}
    \lim_{\delta \to 0}  \lim_{N \to \infty} {\sup_{X_N \in \Lambda_{N}} } \left| {\rm F}_{N}({\rm emp}_{N} , \mu_{\theta}) - \mathcal{E}(\rho - \mu_{\theta}) \right| =0.
\end{equation}

Since the estimates are valid for any $X_{N} \in \Lambda_{N}$, we have that convergence is uniform. 

This concludes the proof of Proposition \ref{Prop:fund}. 

\end{proof}

\section{Appendix: Riesz and log gases at mid temperature}
\label{Rieszplasma}

So far, we have discussed only general interactions at high temperature. It is natural to ask if it is possible to obtain a result valid in a more general temperature regime if we assume additional hypotheses on the interactions. If we specialize to Riesz interactions, then we can indeed obtain such a result. The results obtained in this section are not used in the rest of the paper, but we include them out of independent interest. For this section, we will work on $\mathbb{R}^{d}$, and not on $\mathbb{Q}^{d}$, i.e. We will no longer assume that $V$ is infinite outside of $\mathbb{Q}^{d}$. We start by recalling some well-known facts about Riesz gasses. 

\begin{lemma}[Equilibrium measure]
\label{lem:eqmeas}
Let $g: \mathbb{R}^{d} \to \mathbb{R}$ be given by equations \eqref{Coulombcase} and \eqref{Rieszcase}.

Assume that $V$ satisfies:
\begin{itemize}
    \item[1.] $V$ is lower semi-continuous. 
    
    \item[2.] 
    \begin{equation}
        \lim_{x \to \infty} \frac{V(x)}{2} + g(x) = \infty.
    \end{equation}
    
    \item[3.] $V$ is finite on a set of positive measure. 
\end{itemize}

Then $\mathcal{E}_{V}$ (given by equation \eqref{meanfieldlimit}) has a unique minimizer in the set of probability measures, which we denote $\mu_{V}$. Furthermore, $\mu_{V}$ has compact support, which we denote $\Sigma$, and satisfies
\begin{equation}
    \begin{cases}
    h^{\mu_{V}} + \frac{V}{2} + c \geq 0\\
    h^{\mu_{V}}(x) + \frac{V}{2}(x) + c = 0 \ {\rm for} \ x \in \Sigma
    \end{cases}
\end{equation}
for some $c \in \mathbb{R}$. 
\end{lemma}

\begin{proof}
See, for example, \cite{serfaty2015coulomb}.
\end{proof}

\begin{lemma}[Splitting formula]
Assume that $g$ is given by equations \eqref{Coulombcase} and \eqref{Rieszcase}, and that $V$ satisfies items $1-3$ of Lemma \ref{lem:eqmeas}. Then for any point configuration $X_{N} \in \mathbb{R}^{d \times N}$, the Hamiltonian can be rewritten (split) as 
\begin{equation}
    \mathcal{H}_{N}(X_{N}) = N^{2} \mathcal{E}_{V}(\mu_{V}) + 2N \sum_{i=1}^{N} \zeta(x_{i}) + N^{2}    {\rm F}_{N}(X_{N}, \mu_{V}) ,
\end{equation}
where
\begin{equation}
    \zeta(x) = h^{\mu_{V}} + \frac{V}{2} + c.
\end{equation}
\end{lemma}

\begin{proof}
See, for example, \cite{serfaty2015coulomb}.
\end{proof}

\begin{definition}[Next-order partition function]
We define the mean-field next-order partition function $\widetilde{K}_{N, \beta}$ as 
\begin{equation}
    \widetilde{K}_{N, \beta} = \frac{Z_{N, \beta}}{\exp \left( -N^{2} \beta \mathcal{E}_{V}(\mu_{V}) \right)}.
\end{equation}
\end{definition}

Now that we have recalled some well-known results, we state the main new result of this section. 

\begin{proposition}
\label{Prop:rieszmedtemp}
Let $g$ be given by equations \eqref{Coulombcase} and \eqref{Rieszcase}. Let $ \overline{\mathbf{P}}_{N}$ be as in Definition \ref{def:tagempfiel}, with $\Omega = \Sigma$. Assume that $\beta = N^{-\gamma}$, with $\gamma \in \left( \frac{d-2s}{d}, 1 \right)$ (in the case of equation \eqref{Coulombcase} we take $s=1$ for $d \geq 2$ and $s = \frac{1}{2}$ for $d=1$). Assume hypotheses (H1)-(H5) of \cite{leble2017large}. Then the push-forward of $\mathbf{P}_{N, \beta}$ by $ \overline{\mathbf{P}}_{N}$ satisfies an LDP in $\mathcal{P}(\Sigma \times {\rm Config})$ at speed $N$ with rate function given by   
\begin{equation}
    \mathcal{F}( \overline{\mathbf{P}}) = \begin{cases}
         \mathcal{J}( \overline{\mathbf{P}})  - \left(  {\rm ent}[\mu_{V}] -1 + |\Sigma| \right) \ \ &\textrm{if}\ \  \overline{\mathbf{P}} \in \mathcal{P}_{s,1}(\Sigma \times {\rm Config})\\
       \infty \ \ \ \ \ \ \  \ \ \ \ \ \ \ \ \ \ \ \ \ \ \ \ \ \ \ \ \ \ \ \ \ \ \ \ \ \ \ \ &\textrm{if}\ \  \overline{\mathbf{P}} \notin \mathcal{P}_{s,1}(\Sigma \times {\rm Config})
    \end{cases} 
\end{equation}
where
\begin{equation}
    \mathcal{J}( \overline{\mathbf{P}}) = 
    \begin{cases}
    \overline{\rm Ent}[ \overline{\mathbf{P}} | \overline{\mathbf{\Pi}^{1}}] \ \   \textrm{if} \ \ {\rm int}[ \overline{\mathbf{P}}^{x}]=\mu_{V}(x)\  {\rm a.e.}\\
    \infty \ \ \ \ \ \ \ \ \ \ \ \ \textrm{if not.}
    \end{cases}
\end{equation}

Furthermore, $\widetilde{K}_{N, \beta}$ satisfies
\begin{equation}
    \lim_{N \to \infty} \frac{\log \widetilde{K}_{N, \beta}}{N} \to - {\rm ent}[\mu_{V}].
\end{equation}
\end{proposition}

\begin{proof}
We introduce the constant 
\begin{equation}
    \omega_{N} = \int_{\mathbb{R}^{d}} \exp \left( - 2 N \beta \zeta (x) \right) \, \mathrm d x,
\end{equation}
and also the probability measure 
\begin{equation}
    \rho_{N} (x) = \frac{1}{\omega_{N}} \exp \left( - 2 N \beta \zeta (x) \right).
\end{equation}
Note that the Gibbs measure may be rewritten, for any point configuration $X_{N} \in \mathbb{R}^{d \times N} $ as 
\begin{equation}
    \mathrm d \mathbf{P}_{N, \beta} (X_{N}) = \frac{1}{\widetilde{K}_{N, \beta}} \exp \left( - N^{2} \beta {\rm F}_{N}(X_{N}, \mu_{V}) + N\log \omega_{N} \right) \Pi_{i=1}^{N} \rho_{N}(x_{i}) \mathrm d X_{N}.
\end{equation}

We also introduce the constant 
\begin{equation}
    c_{\omega, \Sigma} = \log |\omega| - |\Sigma| + 1, 
\end{equation}
where
\begin{equation}
    \omega := \{ x \in \mathbb{R}^{d} | \zeta(x)=0 \}.
\end{equation}

Note that if $ \overline{\mathbf{P}}_{N} \to \overline{\mathbf{P}}$ and ${\rm int}[ \overline{\mathbf{P}}^{x}] = \mu_{V}$ a.e. does not hold, then 
\begin{equation}
    \liminf_{N \to \infty} {\rm F}_{N}(X_{N}, \mu_{V}) >0,
\end{equation}
(see Subsection \ref{sect:comp}) and therefore for $\epsilon$ small enough
\begin{equation}
    \lim_{N \to \infty} \frac{1}{N} \log \left( \mathbf{P}_{N, \beta} \left( \overline{\mathbf{P}}_{N} \in B( \overline{\mathbf{P}},\epsilon)  \right)\right) = - \infty.
\end{equation}

On the other hand, if ${\rm int}[ \overline{\mathbf{P}}^{x}] = \mu_{V}$ a.e. then by \cite{leble2017large} we have that, in the $1d$ log case
\begin{equation}
     \liminf_{N \to \infty} N {\rm F}_{N}(X_{N}, \mu_{V}) + {\log N} \geq \overline{\mathbb{W}}( \overline{\mathbf{P}}, \mu_{V}),
\end{equation}
in the $2d$ log case
\begin{equation}
     \liminf_{N \to \infty} N{\rm F}_{N}(X_{N}, \mu_{V}) + \frac{\log N}{2} \geq \overline{\mathbb{W}}( \overline{\mathbf{P}}, \mu_{V}),
\end{equation}
and in all other cases
\begin{equation}
     \liminf_{N \to \infty} N^{\frac{2s}{d}} {\rm F}_{N}(X_{N}, \mu_{V})  \geq \overline{\mathbb{W}}( \overline{\mathbf{P}}, \mu_{V}).
\end{equation}
See \cite{leble2017large} for a definition of $\overline{\mathbb{W}}$. 

Furthermore, by Proposition 4.1 of \cite{leble2017large},
\begin{equation}
    \lim_{\epsilon \to 0} \lim_{N \to \infty} \frac{1}{N} \log \left( \rho_{N}^{\otimes N} \left( \overline{\mathbf{P}}_{N} \in B( \overline{\mathbf{P}},\epsilon)  \right)\right) = -\overline{\rm Ent}[ \overline{\mathbf{P}} | \overline{\mathbf{\Pi}}^{1}]- c_{\omega, \Sigma}.
\end{equation}

Therefore, as long as $\overline{\mathbb{W}}( \overline{\mathbf{P}}, \mu_{V}) < \infty$ we have
\begin{equation}
    \begin{split}
        &\lim_{\epsilon \to 0} \lim_{N \to \infty} \frac{1}{N} \log \left( \mathbf{P}_{N, \beta} \left( \overline{\mathbf{P}}_{N} \in B( \overline{\mathbf{P}},\epsilon)  \right)\right)  \\
       \leq & \lim_{\epsilon \to 0} \lim_{N \to \infty} -\frac{1}{N}\log \left( \widetilde{K}_{N, \beta} \right) -  N \beta \inf_{X_{N}: \overline{\mathbf{P}}_{N} \in B( \overline{\mathbf{P}},\epsilon)}{\rm F}_{N}(X_{N}, \mu_{V}) \\
       &\ \ \ \ \ \ \ \ \ + \log \omega_{N} +  \frac{1}{N} \log \left( \rho_{N}^{\otimes N} \left( \overline{\mathbf{P}}_{N} \in B( \overline{\mathbf{P}},\epsilon) \right)\right)  \\
    \leq & \lim_{N \to \infty} - \frac{1}{N}\log \left( \widetilde{K}_{N, \beta} \right) - \mathcal{J}( \overline{\mathbf{P}}) + |\Sigma| - 1,
    \end{split}
\end{equation}
where we have used that if $N \beta \to \infty$ then
\begin{equation}
    \lim_{N \to \infty} \log \omega_{N} = \log |\omega|
\end{equation}
by Dominated Convergence Theorem. Similarly, if $\overline{\mathbb{W}}( \overline{\mathbf{P}}, \mu_{V}) < \infty$ then by Proposition 4.2 of \cite{leble2017large}, we have that 
\begin{equation}
        \lim_{\epsilon \to 0} \lim_{N \to \infty} \frac{1}{N} \log \left( \mathbf{P}_{N, \beta} \left( \overline{\mathbf{P}}_{N} \in B( \overline{\mathbf{P}},\epsilon) \right) \right) \geq \lim_{N \to \infty} - \frac{1}{N}\log \left( \widetilde{K}_{N, \beta} \right) - \mathcal{J}( \overline{\mathbf{P}}) + |\Sigma| - 1.
\end{equation}

From this, and since stationary tagged point processes $\overline{\mathbf{P}}$ such that $\overline{\mathbb{W}}( \overline{\mathbf{P}}, \mu_{V}) < \infty$ are dense in $\mathcal{P}_{s,1}(\Sigma \times {\rm Config})$ (see \cite{leble2017large}), we conclude that the push-forward of $\mathbf{P}_{N, \beta}$ by $ \overline{\mathbf{P}}_{N}$ satisfies an LDP at speed $N$ with (the non-trivial part of the) rate function given by   
\begin{equation}
    \mathcal{F}( \overline{\mathbf{P}}) = \mathcal{J}( \overline{\mathbf{P}}) - \inf_{\overline{\mathbf{P}}^* \in \mathcal{P}_{s,1}(\Sigma \times {\rm Config}) } \mathcal{J}( \overline{\mathbf{P}}^*),
\end{equation}
and that $\widetilde{K}_{N, \beta}$ satisfies
\begin{equation}
    \lim_{N \to \infty} \frac{\log \widetilde{K}_{N, \beta}}{N} \to |\Sigma|-1 - \inf_{\overline{\mathbf{P}}^* \in \mathcal{P}_{s,1}(\Sigma \times {\rm Config}) } \mathcal{J}( \overline{\mathbf{P}}^*).
\end{equation}

In order to conclude, we note that the minimum of $\mathcal{J}(\overline{\mathbf{P}})$ is achieved at $\overline{\mathbf{P}} = \overline{\mathbf{\Pi}}^{\mu_{V}}$. By \cite{leble2017large}, Lemma 4.4, we have that the minimum is given by 
\begin{equation}
\begin{split}
     \inf_{\overline{\mathbf{P}}^* \in \mathcal{P}_{s,1}(\Sigma \times {\rm Config}) } \mathcal{J}( \overline{\mathbf{P}}^*) &= \overline{\rm Ent}[ \overline{\mathbf{\Pi}}^{\mu_{V}} | \overline{\mathbf{\Pi}}^{1}]\\
     &= \int_{\Sigma} \mu_{V}(x)\log(\mu_{V}(x)) - (\mu_{V}(x) - 1)\, \mathrm d x\\
     &= {\rm ent}[\mu_{V}] -1 + |\Sigma|. 
\end{split}     
\end{equation}
\end{proof}

\begin{remark}
As mentioned before, it is not possible to guess the right rate function by starting from the LDP in \cite{leble2017large} and simply dropping the energy term. However, $\mathcal{J}$ is the pointwise limit of the rate functions as $\theta \to 0$. As mentioned before, $\overline{\mathbb{W}}( \overline{\mathbf{P}}, \mu_{V}) = \infty$ if it is not true that ${\rm int}[ \overline{\mathbf{P}}^{x}] = \mu_{V}(x)$ a.e. This implies that 
\begin{equation}
   \lim_{\theta \to 0} \theta \overline{\mathbb{W}}( \overline{\mathbf{P}}, \mu_{V}) + \overline{\rm Ent}[ \overline{\mathbf{P}} | \overline{\mathbf{\Pi}}^{1}] = \mathcal{J}( \overline{\mathbf{P}})
\end{equation}
for all $ \overline{\mathbf{P}} \in \mathcal{P}_{s}(\Sigma \times {\rm Config})$.
\end{remark}

\section{Acknowledgements}

DPG acknowledges support by the German Research Foundation (DFG) via the research unit FOR 3013 “Vector- and tensor-valued surface PDEs” (grant no. NE2138/3-1). {We thank Sylvia Serfaty, Ed Saff, and Gaultier Lambert for useful conversations}.\\

\textbf{Data Availibility Statement:} All data generated or analysed during this study are included in this published article.

\section*{Declarations}

\textbf{Conflict of interest:} This work has no conflicts of interest.

\bibliographystyle{plain}
\bibliography{bibliography.bib}

\end{document}